\documentclass[final]{siamltex}

\usepackage{amsfonts}
\usepackage{amsmath}
\usepackage{graphicx}
\usepackage{subfigure}
\usepackage[mathscr]{euscript}

\newcommand{\real}{\mathbb{R}}

\newcommand{\bfE}{\mathbb{E}}
\newcommand{\bP}{\mathbb{P}}

\newcommand{\sF}{\mathcal{F}}

\newcommand{\posint}{\mathbb{Z}_+}

\newtheorem{assume}{Assumption}
\newtheorem{remark}{Remark}[section]

\title{Efficiency of the Girsanov transformation approach for parametric sensitivity analysis of stochastic chemical kinetics}

\author{Ting Wang\thanks{Department of Mathematical Sciences, University of Delaware,
        DE 19716 ({\tt tingw@udel.edu}).}
        \and Muruhan Rathinam\thanks{Department of Mathematics and Statistics, University of Maryland, Baltimore County,
        MD 21250 ({\tt muruhan@umbc.edu}).}}

\begin{document}
\maketitle
\newcommand{\slugmaster}{%
\slugger{juq}{xxxx}{xx}{x}{x--x}}

\begin{abstract}
Most common Monte Carlo methods for sensitivity analysis of 
stochastic reaction networks are the finite difference (FD), the Girsanov
transformation (GT) and the regularized pathwise derivative (RPD) methods. 
It has been numerically observed in the literature, that the
biased FD and RPD methods tend to have lower variance than the 
unbiased GT method and that centering the GT method 
(CGT) reduces its variance. We provide a theoretical justification for these 
observations in terms of system size asymptotic analysis under what is known
as the classical scaling. 
Our analysis applies to GT, CGT and FD, and shows that the standard 
deviations of their estimators when normalized by the actual sensitivity, 
scale as $\mathcal{O}(N^{1/2}), \mathcal{O}(1)$ and
$\mathcal{O}(N^{-1/2})$ respectively, as system
size $N \to \infty$. In the case of the FD methods, the $N \to \infty$ asymptotics 
are obtained keeping the finite difference perturbation $h$ fixed. 
Our numerical examples verify that our order estimates are sharp
and that the variance of the RPD method scales similarly to the FD methods. 
We combine our large $N$ asymptotics with previously known small $h$
asymptotics to obtain the best choice of $h$ in terms of $N$, 
and estimate the number $N_s$ of simulations required to achieve 
a prescribed relative $\mathcal{L}_2$ error $\delta$. This shows that $N_s$ depends on 
$\delta$ and $N$ as
$\delta^{-2 - \frac{\gamma_2}{\gamma_1}} N^{-1}, \delta^{-2}$ and $N \delta^{-2}$,
for FD, CGT and GT respectively. Here $\gamma_1 >0, \gamma_2>0$
depend on the type of FD method used.  
\end{abstract}

\begin{keywords}
stochastic chemical kinetics, Girsanov transformation, asymptotic analysis,
parametric sensitivity, finite difference, variance analysis. 
\end{keywords}

\begin{AMS}
Primary: 60H35, 65C99; Secondary: 92C42, 92C45
\end{AMS}

\pagestyle{myheadings}
\thispagestyle{plain}
\markboth{Ting Wang AND Muruhan Rathinam}{Efficiency of the Girsanov transformation approach}

\section{Introduction}
Estimation of parametric sensitivities of dynamical systems is an
essential part of the modeling and parameter estimation process. 
For instance the problem of finding the set of parameters that best fit 
some observed data can be formulated as an optimization problem over 
the parameter space where the partial derivatives of the objective function
depend on the parametric sensitivities defined as partial derivatives of some 
system output with respect to the parameters.  

In deterministic dynamical systems governed by ordinary differential equations (ODEs), the sensitivities 
defined by the partial derivatives  $\partial f(X(t)) /\partial c_k$, of some function $f$ of the state with respect to the parameters 
are essentially computed by numerical integration of an auxiliary system of 
evolution equations obtained by linearization of the original ODEs. 
In contrast, for stochastic dynamical systems, several vastly different approaches
exist. We note that we shall treat the
parameters $c_k$ as deterministic and not as random quantities, while the dynamic
behavior of the systems we consider is stochastic. 

Our primary focus will be stochastically modeled chemical reaction systems. While the stochastic 
chemical kinetic model under the well stirred assumption \cite{Gillespie77} has been around for decades,
it wasn't until the late nineties that the importance of stochastic chemical models in some applications was realized
\cite{arkin1998stochastic, mcadams1999sa}. Especially, intracellular chemical reactions systems, often contain certain molecular species in small copy numbers, 
and as such, the deterministic model based on ordinary differential equations (ODEs) or partial differential 
equations (PDEs) for the concentrations of various molecular species is not appropriate. 
A more appropriate model, under the well stirred assumption, consists of a continuous time Markov process
$X(t)$ with the nonnegative integer lattice $\mathbb{Z}_+^n$ as state space.    

While we focus on stochastic chemical kinetics which we describe in the next subsection, we note
that analogous models appear in other fields such as epidemiology and 
predator-prey systems.
  
\subsection{Stochastic chemical kinetics}
As a simple example, let us consider the chemical reaction system 
\begin{equation}\label{eq:S1S2S3}
S_1 + S_2 \rightarrow S_3, \;\; S_3 \rightarrow S_1 + S_2,
\end{equation}
consisting of three species $S_1, S_2$ and $S_3$ undergoing two reaction channels. The state space 
is the set $\mathbb{Z}_+^3$ of nonnegative three dimensional integer vectors, where the state $x=(x_1,x_2, x_3)$ 
describes the copy numbers $x_1$ of $S_1$, $x_2$ of $S_2$ and $x_3$ of $S_3$.
When the first reaction channel fires, the state changes by $\nu_1=(-1,-1,1)^T$, and when the second reaction channel fires it changes by $\nu_2=(1,1,-1)^T$. The quantities $\nu_j$ are known as {\em stoichiometric vectors}
and for chemical reaction systems the $\nu_j$ are parameters and state independent. The ``probabilistic rate'' at which these two reactions occur is given by
the {\em intensity functions} $a_1(x,c)$ and $a_2(x,c)$ (where $c$ is a vector of parameters). The precise meaning of the intensity functions is as follows. 
If $X(t)=(X_1(t),X_2(t), X_3(t))$ is the stochastic process of species counts, then 
given $X(t)=x$, the probability of at least one firing of the $j$th
reaction channel during interval $(t,t+h]$ is $a_j(x,c) h + o(h)$ as $h \to
  0^+$.

{\bf Stochastic mass action form:}
Under the well stirred model of Gillespie \cite{Gillespie77}, the intensity functions take the 
{\em stochastic mass action} form:
$a_1(x,c) = c_1 x_1 x_2$ and $a_2(x,c) = c_2 x_3$. 
The rationale for this 
specific form is based on the following considerations. The probability that a given pair of one $S_1$ and one $S_2$ 
molecules come together and react during time interval $(t,t+h]$ is given by $c_1 h + o(h)$ where $c_1$ is a constant. Given that there are $x_1 x_2$ different ways to choose the pair, we obtain the probability of $c_1 x_1 x_2 h + o(h)$ 
for any pair of $S_1$ and $S_2$ to react. Likewise, the probability that a given $S_3$ molecule gives rise to an $S_1$ and an $S_2$ via the second reaction during $(t,t+h]$ is given by $c_2 h + o(h)$ where $c_2$ is a constant. Given that there are $x_3$ different $S_3$ molecules, we obtain the probability of $c_2 x_3 h + o(h)$ 
for any of the $S_3$ to react.

{\bf General chemical system:}
More generally, a chemical reaction system consists of $m$ reaction channels and $n$ chemical species $\{S_1, \cdots, S_n\}$.
The $n$-dimensional state vector $X(t)$ characterizes the state of the system
where each entry $X_i(t)$ represents the number of molecules of the species
$S_i$ at time $t$. 
The firing of a reaction channel $j \in \{1,\cdots,m\}$ at time $t$ causes the state to be incremented by the 
stoichiometric vector $\nu_j$.  We assume that $X$ is {\em c\`adl\`ag}, i.e.\ paths of $X$ are right continuous with left-hand limits and hence, if reaction channel $j$ fires at time $t$, then $X(t)=X(t-)+\nu_j$. 
For $j=1,\cdots,m$ we denote by $R_j(t)$ the number of firings of the $j$-th
reaction channel during $(0,t]$. Thus $X(t) = X(0) + \nu R(t)$ for $t \geq 0$,
  where $\nu$ is the {\em stoichiometric matrix} whose $j$th column is $\nu_j$
  and $R(t)=(R_1(t),\cdots,R_m(t))^T$. We note that $R(0)=0$ and
  $R_j(t)-R_j(t-)$ is either $0$ or $1$. 
The process $X$ is assumed to be Markovian, and
associated with each reaction channel is an {\em intensity function} (also known as {\em propensity function} in the chemical kinetics literature) $a_j(x,c), j = 1, \cdots, m$, which is such that, given $X(t)=x$ the probability of
one or more firing of reaction channel $j$ during $(t,t+h]$ is $a_j(x,c) h+
  o(h)$ as $h \to 0^+$. Here, $c$ are parameters. 
Following the terminology of \cite{Bremaud}, we note 
that $R_j$ are counting processes which admit the $\sF_t$-predictable
intensity process $a_j(X(t-),c)$ where $\sF_t$ is the filtration generated by
$X$ and $R$.

{\bf Random time change representation:}
Naturally, the probability laws of the stochastic processes $X$ and $R$,
depend on the parameters $c$. For the purpose of analyses, it proves
convenient to find a way to represent the processes $X$ and $R$ corresponding
to different $c$ values on the same sample space $(\Omega,\sF,\bP)$.  
To this end, we use the {\em random time change representation} \cite{Ethier-Kurtz} to express $X$ via 
the stochastic equation
\begin{equation}\label{equ:RTC}
X(t,c) = x_0 + \sum_{j = 1}^{m} Y_j\left(\int_0^t a_j(X(s,c),c) ds\right) \nu_j
\end{equation}
where $Y_j$ are independent unit rate Poisson processes. It follows that 
\begin{equation}
R_j(t,c) = Y_j\left(\int_0^t a_j(X(s,c),c) ds\right), \quad j=1,\dots,m,
\end{equation}
where $x_0$ is the initial state assumed to be deterministic. 
We note that in this representation, 
we have a family of stochastic processes $X(t,c)$ and $R(t,c)$ on the same
sample space $(\Omega,\sF,\bP)$ where each element $\omega \in \Omega$ may be identified with a specific 
trajectory of $Y(t)=(Y_1(t),\dots,Y_m(t))$, the underlying unit rate independent Poissons. We note that the $Y_j(t)$ do not depend on the parameters. See \cite{CRP} for a detailed explanation of how to compute $X(t,c)$ once 
a sample path of $Y(t)$ is generated.  
   
\subsection{Parametric sensitivity estimation}
We consider parametric sensitivities of the stochastic process $X(t,c)$ with respect to
an output function $f:\mathbb{Z}_+^n \to \real$,  
defined by the partial derivatives 
\[ \frac{\partial}{\partial c_k} \mathbb{E} (f(X(t,c))),\]
where $c_k$ are scalar parameters, $f$ is some suitable scalar function of the
state space, $\mathbb{E}$ is the expectation and $t>0$ is some fixed final time. For simplicity we shall 
focus on one scalar parameter $c$. 
When the number of species $n$ is large (in several applications it is of the order of $10-100$), 
due to the curse of dimensionality, Monte Carlo approaches are the most viable for both simulation of 
the process $X$ as well as estimation of sensitivities. Monte Carlo simulation of 
exact sample paths of the process $X$ is feasible and is provided by the well known SSA or Gillespie algorithm
\cite{Gillespie77}. In this context several different Monte Carlo approaches exist for 
the numerical computation of the parametric sensitivities as well.

We shall use $\mathscr{S}(t, c)$ to denote the exact sensitivity
\begin{equation}
\mathscr{S}(t, c) = \frac{\partial}{\partial c} \mathbb{E}(f(X(t, c))).
\end{equation}
As we will see later in this section,
all the Monte Carlo methods for computing the sensitivity involve the 
estimation of the expected value $\bfE(S(t,c))$ of some process $S(t,c)$ at 
time $t>0$, via i.i.d.\ sample estimation, where $S(t,c)$ can be computed easily from 
the knowledge of system parameters, the function $f$ and the 
sample path of $X$ on the time interval $[0,t]$. In other words, one generates $N_s$ independent copies $X^{(i)}(t,c)$ of $X(t,c)$ for $i=1,\dots,N_s$, and 
then computes the corresponding copies $S^{(i)}(t,c)$ of $S(t,c)$.
Then the sensitivity is estimated by
\[
\bar{S}(t,c) = \frac{1}{N_s} \sum_{i=1}^{N_s} S^{(i)}(t,c).
\]
Since $\bfE(\bar{S}(t,c))=\bfE(S(t,c))$ and $\text{Var}(\bar{S}(t,c))=\text{Var}(S(t,c))/N_s$, the accuracy of this
estimate depends on the error (known as bias) $\mathbb{E}(S(t, c) - \mathscr{S}(t,c))$, the variance $\text{Var}(S(t,c))$ of the underlying estimator $S(t,c)$ and the sample size $N_s$. 

One way to quantify the error in estimation is via the mean square error:
\begin{equation}
\mathbb{E}\left(|\bar{S}(t,c) - \mathscr{S}(t, c)|^2\right) = \frac{\text{Var}(S(t,c))}{N_s} + (\bfE (S(t, c) - \mathscr{S}(t, c)))^2.
\end{equation}
If $\text{Var}(S(t,c))$ is large, then one requires greater number $N_s$ of
simulations resulting in loss of efficiency. On the other hand if a biased
estimator is used, increasing the number of simulations $N_s$ does not
help. It is often useful to consider the {\em relative error} (RE) defined by
\begin{equation}\label{eq-RE}
\text{RE} = \sqrt{\bfE\left(|\bar{S}(t,c) - \mathscr{S}(t,
  c)|^2\right)}/|\mathscr{S}(t, c)|,
\end{equation}
provided the true sensitivity $\mathscr{S}(t, c)$ is nonzero. 

Throughout this paper, we shall refer to $S(t,c)$ as the {\em underlying estimator} or simply the {\em
  estimator}  and 
$\bar{S}(t,c)$ as the {\em ultimate estimator}. As the properties of the
latter depend directly on that of the former and $N_s$, the analysis of the 
variance of the underlying estimator $S(t,c)$ shall be our focus.
We define the {\em relative standard deviation} (RSD) and
the {\em relative bias} (RB) of the underlying estimator $S(t,c)$ by
\begin{equation}\label{eq-RSD-def}
\text{RSD} = \sqrt{\text{Var}(S(t,c))}/|\mathscr{S}(t, c)|,
\end{equation}
and
\begin{equation}\label{eq-RB-def}
\text{RB} = \bfE(S(t,c) - \mathscr{S}(t,c))/|\mathscr{S}(t, c)|,
\end{equation}
when $\mathscr{S}(t, c) \neq 0$.
We note that the relative error is given by
\begin{equation}\label{eq-RE-RSD-RB}
\text{RE} = \sqrt{\frac{\text{RSD}^2}{N_s} + \text{RB}^2}.
\end{equation}

Now, we turn our attention to the description of some common Monte Carlo 
sensitivity estimators. As a general reference on this topic we suggest \cite{Glynn-book, Glasserman-book}. 
The Monte Carlo methods for sensitivity can
broadly be categorized into finite difference (FD) methods \cite{CFD,Glynn-book,CRP}, pathwise derivative
(PD) methods \cite{Glynn-book,RPD} and the likelihood ratio or the Girsanov transformation (GT)
methods \cite{Glynn-book,Gir}. 

The FD methods involve approximation of the partial derivative by
the simple finite difference $\bfE [f(X(t,c+h)) - f(X(t,c))]/h$ 
or some higher order finite difference.  
In the case of the simple FD above, 
\begin{equation}\label{eqn:FD-estimator}
S_{\text{FD}}(t,c) = h^{-1}[f(X(t,c+h)) - f(X(t,c))].
\end{equation}
Thus, $\bfE(S_{\text{FD}}(t,c)) \neq \frac{\partial}{\partial c} \mathbb{E} (f(X(t,c)))$ 
in general, and the bias is decreased by decreasing $h$.
On the other hand,
\begin{equation*}
\begin{split}
\text{Var}(S_{\text{FD}}(t,c)) = h^{-2} \left\{\text{Var}(f(X(t,c+h))) + \text{Var}(f(X(t,c))) 
-2 \text{Cov}\left(f(X(t,c+h)),f(X(t,c))\right)\right\}.
\end{split}
\end{equation*}
In general the numerator does not vanish as fast as
$h^2$ when $h \to 0$, showing that small $h$ leads to large variance. 
When $f(X(t,c+h))$ and $f(X(t,c))$ are strongly positively correlated, one
may expect the variance to be small. 
If the processes $X(t,c)$ and $X(t,c+h)$ are taken to be independent, which is 
accomplished by the use of two independent streams of random numbers in the
simulation, the resulting FD method is known as the {\em independent random
  number} (IRN) method. If the processes $X(t,c)$ and $X(t,c+h)$ are strongly
coupled, which is accomplished by the use of a common random number stream, 
the resulting approach is known as {\em common random number} (CRN) method.
In general, the CRN FD methods have much lower variance than the IRN FD
methods.  
Moreover, different approaches to couple the processes $X(t,c+h)$ and $X(t,c)$ lead to different covariances and hence different variances for the FD estimators. 
See \cite{CFD, CRP} for some approaches.  

In the PD method one takes
\[
S_{\text{PD}}(t,c) = \frac{\partial}{\partial c} f(X(t,c)),
\]
and the method is applicable provided the derivative exists, analytical computation of the
derivative is possible and the commutation
\begin{equation}\label{eq-PD}
\bfE \left(  \frac{\partial}{\partial c} f(X(t,c)) \right) =
\frac{\partial}{\partial c} \bfE (f(X(t, c)))
\end{equation}
holds. 
In the context of stochastic chemical kinetics, direct application of the PD method is not valid as the commutation in \eqref{eq-PD} does not hold.  To see this, note that $f(X(t,c,\omega))$ is piecewise constant 
in $c$ for fixed $t$ and $\omega$ and hence the derivative is $0$, while the sensitivity $\partial \bfE (f(X(t, c))) / \partial c$ is in
general non-zero, showing that the commutation in \eqref{eq-PD} is not valid (see \cite{RPD} for details).  
It is possible to regularize the problem by replacing
$\partial  f(X(t,c)) / \partial c$ with
\begin{equation}\label{eq-S-RPD}
S_{\text{RPD}}(t,c) = \frac{\partial}{\partial c} \left(\frac{1}{2 w} \int_{t-w}^{t+w} f(X(s,c)) ds\right),
\end{equation}
to obtain the {\em regularized pathwise derivative} (RPD) estimator for which the commutation of derivative with
expectation holds for a restricted class of examples \cite{RPD}. 
This, however results in a bias which increases with large $w$. 
Also see \cite{Glasserman-book} for similar work 
 in the context of computing the sensitivity of 
path integrals. 

The GT approach may be motivated in different ways. For the purpose of our analysis based on the random time change representation, it is natural to start with the family of processes $X(t,c)$ parametrized by $c$ that 
are all defined on $(\Omega,\sF,\bP)$ as mentioned before. 
Suppose the sensitivity is required at a specific parameter value $c=c_0$. Under certain regularity conditions,
a family of new probability measures $P(c)$ may be constructed on the same sample space $(\Omega,\sF)$ for a range of $c$ values in a neighborhood of $c_0$ so that $P(c_0)=\bP$, i.e.\ coincides with the original 
probability measure (see \cite{Bremaud} for instance). Moreover, the probability measures $P(c)$ are absolutely continuous with respect to 
$P(c_0)$ and the $P(c)$-law of the process $X(t,c_0)$ is the same as the $P(c_0)(=\bP)$-law of the process $X(t,c)$.
In other words, for all suitable functions $f$,
\[
\int_\Omega f(X(t,c)) dP(c_0) = \int_\Omega f(X(t,c_0) dP(c).
\]
We observe that the left hand side is $\bfE(f(X(t,c)))$. If we denote by $L(t,c,c_0)$ the Radon-Nikodym derivative $dP(c)/dP(c_0)$, then we have
\begin{equation}\label{eq-GT}
\begin{aligned}
\left.\frac{\partial}{\partial c}\right|_{c=c_0} \bfE (f(X(t,c))) &= \left.\frac{\partial}{\partial c}\right|_{c=c_0}
\int_{\Omega} f(X(t,c_0)) L(t,c,c_0) dP(c_0)\\
&= \int_{\Omega} f(X(t,c_0)) \left.\frac{\partial}{\partial c}\right|_{c=c_0} L(t,c,c_0) dP(c_0) 
\end{aligned}
\end{equation}
provided the differentiation inside the integral is valid. It turns out that
\begin{equation}\label{eq:def-Z}
Z(t,c_0) = \left.\frac{\partial}{\partial c}  \right|_{c=c_0} L(t,c,c_0),
\end{equation}
is analytically tractable and the required sensitivity is given by
\[
\left.\frac{\partial}{\partial c}\right|_{c=c_0} \bfE (f(X(t,c))) = \bfE [f(X(t,c_0)) Z(t,c_0)],
\]
thus the sensitivity estimator $S(t,c_0) = f(X(t,c_0)) Z(t,c_0)$. 

In the context of stochastic chemical kinetics, the weight process $Z$ defined by \eqref{eq:def-Z} is given by \cite{Gir, RPD}
 \begin{equation}\label{eq:Z}
 \begin{split}
 Z(t,c)  = \sum_{j = 1}^m \int_0^t \frac{\frac{\partial a_j}{\partial c}(X(s-,c), c)}{a_j(X(s-,c), c)} dR_j(s,c) - \sum_{j = 1}^m \int_0^t\frac{\partial a_j}{\partial c}(X(s,c), c) ds.
\end{split} 
 \end{equation}
We have dropped $c_0$ in favor of $c$ for notational ease, however, it must be noted that all computations are carried out at the specific parameter value $c$ at which the sensitivity is required. 

We also investigate a modified GT method inspired by the work in \cite{PB}, which we call the {\em centered Girsanov transformation} (CGT)
method in which we replace the estimator $f(X(t,c))Z(t,c)$ with $(f(X(t,c))-\mathbb{E}(f(X(t,c))))Z(t,c)$. 
Since $Z(t,c)$ has zero mean this new estimator has the same mean as the
original one and hence is also unbiased. In practice $\bfE(f(X(t,c)))$ is not known
and needs to be estimated as well. One approach would be to generate $N_s$
independent copies  $X^i(t,c)$ of $X(t,c)$ and then use 
\[
\overline{f(X(t,c))} = \frac{1}{N_s} \sum_{i=1}^{N_s} f(X^{(i)}(t,c)),
\]
to estimate $\bfE(f(X(t,c)))$ and then use 
\[
\bar{S}_{\text{CGT}} = \frac{1}{N_{s}} \sum_{i=1}^{N_{s}}
\left( f(X^{(i)}(t,c))-\overline{f(X(t,c))} \right) Z^{(i)}(t,c),
\]
as the ultimate estimator. In this case $\bfE(\bar{S}_{\text{CGT}}) \neq
\bfE(f(X(t,c)) Z(t,c))$ and the estimator is biased. However, when $N_s$ 
is large the bias is $\mathcal{O}(1/N_s)$. Also 
\[
\text{Var}(\bar{S}_{\text{CGT}}) = \text{Var}(S_{\text{CGT}}) /N_s +
\mathcal{O}(1/N_s^2),
\]
where $S_{\text{CGT}} = (f(X(t,c))-\bfE(f(X(t,c)))) Z(t,c)$ is the underlying
CGT estimator. 
So it is adequate to study the variance of $(f(X(t,c))-\bfE(f(X(t,c))))
Z(t,c)$.  
In the formula used in \cite{PB} for the ultimate estimator, 
$Z^{(i)}$ above were replaced by 
$Z^{(i)} - \bar{Z}$ where $\bar{Z}$ was the sample mean of $Z^{(i)}$. When the sample
size $N_s$ is large, both ultimate estimators are similar.   
For the purpose of analysis, we shall focus on the underlying CGT estimator 
\begin{equation}\label{eq-S-CGT-def}
S_{\text{CGT}} = f(X(t,c)) Z(t,c) - \bfE(f(X(t,c))) Z(t,c).
\end{equation}

We note that the variances of the GT and CGT estimators are given by the
following formulae:
\begin{equation}\label{eq:GT-CGT}
\begin{aligned}
\text{Var}(S_{\text{GT}}) &= \bfE((f(X(t,c)))^2 Z^2(t,c)) - \bfE^2(f(X(t,c))
Z(t,c)),\\
\text{Var}(S_{\text{CGT}}) &= \text{Var}(S_{\text{GT}}) - 2 \bfE(f(X(t,c))
Z^2(t,c)) + \bfE^2(f(X(t,c))) \bfE(Z^2(t,c)).\\
\end{aligned}
\end{equation}
It must be noted that it is not always the case that
$\text{Var}(S_{\text{GT}})$ is greater than or equal to
$\text{Var}(S_{\text{CGT}})$. Thus, one cannot conclude that CGT is always 
superior to GT. However, it was observed in \cite{PB} as well as in our
simulations that CGT tends to have lower variance than GT in most examples.

Recently introduced methods, \emph{auxiliary path algorithm (APA)} \cite{APA}
and \emph{Poisson path algorithm (PPA)} \cite{PPA}, do not strictly belong to 
these three categories mentioned above. While they are closely related to the 
FD and the PD methods, they provide unbiased estimators similar to the GT. 
We do not investigate these methods in this paper. 

It has been observed that the PD method, when applicable, yields an estimator 
with lower variance than the GT estimator which is applicable in most
situations \cite{ Glynn-book, RPD}. In the context of stochastic chemical kinetics, 
the regularized PD (RPD) method is only applicable to a limited class of examples and results in a 
biased estimator \cite{RPD}. The FD methods also result in biased estimators. Both the FD
and RPD methods also involve the use of method parameters,  $h$ or $w$, and 
the smaller these are the less the bias of these methods. However, decreasing
$h$ or $w$ results in an increase in the variance of the FD or RPD estimators
respectively. The GT estimator on the other hand is unbiased and does not
involve method parameters to be determined. However, it has been observed that
in many situations, the GT estimator has much larger variance compared to the 
FD and RPD estimators \cite{Glynn-book,Gir, CRP,RPD}. 
To our knowledge, no theoretical explanation has been presented for the large variance of 
the GT method observed in many applications. In this paper, we provide a theoretical 
explanation for the large variance.

\begin{remark}\label{rem-c-zero}
If a coefficient $c_j=0$ in the stochastic mass action form of intensity
functions, then reaction channel $j$ is absent. However, one may want to 
compute the sensitivity at $c_j=0$ to see the effect of ``turning off'' a 
reaction channel. In this case the GT or CGT methods does not work, in fact the 
weight process $Z$ is undefined. 
However, the FD methods 
work. Given the dependence of $Z$ on $c_j$, one also expects the variance of 
$Z$ to approach infinity as $c_j \to 0$. This was numerically examined in
\cite{APA}. 
\end{remark}

\subsection{System size dependence in stochastic mass action}\label{sec-system-size}
In stochastic chemical kinetics as well as other 
population models, there is a ``system size parameter" $N$ and in the $N \to \infty$ 
these systems behave deterministically (see Chapter 11 of \cite{Ethier-Kurtz}
for instance). Our analysis shows that the variance of the GT 
method grows much faster in $N$ than the variances of the FD methods.  

We describe the general stochastic mass action form of intensities that 
commonly arise in stochastic chemical kinetics \cite{Gillespie77} and 
describe how system size $N$ enters into the model. 
If we divide the stoichiometric vector $\nu_j$ into two parts, such that $\nu_j = \nu'_j - \nu''_j$, where
\begin{enumerate}
\item[$\nu'_j $]:  the vector number of molecules of each species that are created in the $j$th reaction,
\item[$\nu''_j$]: the vector number of molecules of each species that are consumed in the $j$th reaction,
\end{enumerate}
then the intensity of the $j$th reaction is
\begin{equation}\label{equ:mass action}
a_{j}^{N}(x, c) = \frac{c_j}{N^{|\nu''_j|-1}}\prod_{i}^{n}\binom{x_i}{\nu''_{ij}},
\end{equation}
where $|\nu''_j| = \sum_{i=1}^{n}\nu''_{ij}$ and $N$ is the volume of
the system times Avogadro's number, $c_j$ is a constant specifying the rate of
the reaction. We note that the term $\binom{x_i}{\nu''_{ij}}$ represents the number of 
ways to choose $\nu''_{ij}$ molecules from $x_i$ molecules of the $i$th
species.
The term $1/N^{|\nu''_j|-1}$ also plays a critical role. To understand this,
let us return to the example in \eqref{eq:S1S2S3}. Let us relabel the
parameters as $c_1'$ and $c_2'$. As $c_1'h + o(h)$ is the probability 
that a given pair of $S_1$ and $S_2$ interact during $(t,t+h]$, one expects 
$c_1'$ to depend on the system volume or equivalently on system size $N$ in inverse
proportion: $c_1'= c_1/N$. Here, the newly defined $c_1$ is independent of 
system size $N$.  On the other hand, for the monomolecular reaction, 
the probability $c_2' h + o(h)$ of a given $S_3$ molecule reacting during
$(t,t+h]$ is independent of system size $N$. In general, when $|\nu''_{j}|$ 
number of molecules come together to react, the term $c_j'$ will depend on 
system size $N$ as
\begin{equation}\label{eqn:cj-cj-prime}
c_j' = c_j / N^{|\nu''_{j}|-1}. 
\end{equation}
See \cite{Gillespie77} for 
more details. It must be noted that it is often useful to model ``pure production'' 
reactions, represented by an abstract chemical equation as $\emptyset \to S$, 
and the stochastic chemical models in literature often utilize such
reactions. In this case, the stochastic mass action form of intensity function
is a constant $c'$ and it is natural to take its dependence on $N$ to be
proportional: $c' = c N$, still satisfying the formula $c'_j = c_j /
N^{|\nu''_{j}|-1}$. 
  
Thus the intensity functions $a_j^N$ depend on $N$ and $x$ in a
specific manner, referred to as {\em density dependence}
(see Chapter 11 of \cite{Ethier-Kurtz}). This density dependence leads to a deterministic
 limiting behavior in the large system size ($N \to \infty$), when the 
initial conditions are also scaled by $N$ so that the initial species counts 
per volume (concentration) is held constant. The relevant theorem from
\cite{Ethier-Kurtz} will be restated in the next
section.

{\bf The parameters $c'_j$ and $c_j$:}
We note that the parameters $c'_j$ (which depend on $N$) are sometimes referred to as the {\em
  stochastic parameters} while $c_j$ are referred to as the {\em deterministic
  parameters}. In practice, one works with $c'_j$, and hence the sensitivities 
with respect to $c'_j$ will be relevant. The sensitivities 
with respect to $c_j$ are related to those with respect to $c'_j$ via 
\begin{equation}\label{eq:sens-c-cprime}
\mathscr{S}_j(t,c) = \frac{\partial}{\partial c_j}\bfE(f(X(t))) =  \frac{\partial}{\partial
  c'_j}\bfE(f(X(t))) N^{1-|\nu''_{j}|} = \mathscr{S}'_j(t,c) N^{1-|\nu''_{j}|}.
\end{equation}
Moreover, if $S$ is a sensitivity estimator for the sensitivity with respect
to the deterministic parameter $c_j$, then $S' = S N^{|\nu''_{j}|-1}$ is a 
sensitivity estimator for the sensitivity with respect
to the stochastic parameter $c'_j$. While the variances and biases of 
the stochastic and deterministic sensitivity estimators scale differently 
with system size $N$, the relative quantities RE, RSD and RB, will scale 
the same way. Therefore, without loss of generality, in the rest of the paper, we shall only concern 
ourselves with sensitivities with respect to the deterministic parameters
$c_j$.

Finally, we like to note that in the stochastic mass action form of intensity
functions, there is precisely one (deterministic) parameter $c_j$ for each 
intensity function $a_j$ and the parameters enter multiplicatively. 
Hence $\frac{\partial a_j}{\partial c_j}/{a_j}  = 1/c_j$.
This leads to the simple form for the weight process $Z(t,c)$ for the
sensitivity with respect to $c_j$
\begin{equation}\label{eq-Z-simple}
Z(t,c) = \frac{1}{c_j} \left( R_j(t,c) - \int_0^t a_j(X(s,c)) ds\right).
\end{equation} 

\subsection{An illustrative example}
To investigate the estimator variance for the GT, CGT and FD methods,
we consider the analytically tractable birth death model from population dynamics, which also appears 
in gene regulatory networks where mRNA is produced at a constant probabilistic
rate and decays at a rate proportional to the number of mRNA. The model
is described by
\begin{equation}\label{eq:birthdeath}
\emptyset \xrightarrow{c_1} S, \; \; S \xrightarrow{c_2} \emptyset.
\end{equation}
The intensity functions are $a_1^N(x,c) = Nc_1$ and $a^N_2(x,c) = c_2 x$.
We consider the output function $f(x)=x$. Denoting by $X^N$ the 
system size dependence of the process, it can be 
shown that 
\begin{equation}\label{eq-EXN-bd}
\bfE(X^N(t,c)) = N x_0 e^{-c_2 t} + \frac{Nc_1}{c_2} (1 - e^{-c_2 t}),
\end{equation}
where we have chosen a deterministic initial condition $X^N(0)=N x_0$. 
The sensitivities with respect to $c_1$ and $c_2$ are given by
\[
\begin{aligned}
\frac{\partial}{\partial c_1}\bfE(X^N(t,c)) &= \frac{N}{c_2} (1 - e^{-c_2t}),\\
\frac{\partial}{\partial c_2}\bfE(X^N(t,c)) &= - N x_0 t e^{-c_2 t} -
\frac{N c_1}{c_2^2} (1 - e^{-c_2 t}) + \frac{N c_1}{c_2} t e^{-c_2 t}.\\
\end{aligned}
\]
We observe that the both sensitivities are $\mathcal{O}(N)$ as $N \to \infty$. Also, in terms of $t$ both
sensitivities are $\mathcal{O}(1)$ as $t \to \infty$.

To study the variance of the GT and CGT estimators, first we consider the sensitivity  $\frac{\partial}{\partial c_1}
\bfE(X^N(t,c))$. The population process $X^N(t,c)$ and the weight process $Z^N(t,c)$
in this case can be written as
\begin{equation}
\begin{split}
X^N(t,c) &= N x_0 - \int_{(0, t]} \,dR^N_1(s,c) + \int_{(0,t]} \, dR^N_2(s,c),\\
Z^N(t,c) &= \int_{(0,t]}\frac{1}{c_1}dR^N_1(s,c) - N\int_{0}^{t} ds,
\end{split}
\end{equation}
where $R^N_j$ and $Z^N$ show dependence on $N$. 
One can use the Ito formula for processes driven by finite variation processes
(see \cite{Rogers}) to write down the stochastic equations for
$(X^N)^\alpha(t,c)(Z^N)^\beta(t,c)$, for the integer powers $0 \leq \alpha, \beta \leq
2$,
and then take expectations to obtain a coupled system of linear ODEs for 
$\bfE((X^N)^\alpha(t,c)(Z^N)^\beta(t,c))$. Then the variance of GT and CGT estimators
can be computed by the relations \eqref{eq:GT-CGT} with $f(x)=x$.

After lengthy calculations with the aid of Maple symbolic software one can show 
that 
\begin{equation}
\begin{split}
\text{Var}(S_{\text{GT}}) = \frac{N{{e}^{-2\,{c_2}\,t}}}{c_1c_2^2}
(&{e^{2\,c_2\,t}}N^2c_1^2t
+N{c_1}\,t{c_2}\,{e^{2\,{c_2}\,t}}
+2\,{e^{{c_2}\,t}}N^2{c_1}\,{c_2}\,t{x_0}
+{{e}^{c_2\,t}}c_2^2tN{x_0}\\
&+c_2^2t{N}^{2}x_0^2
-2\,{{e}^{c_2\,t}}N^2c_1^2t
-{{e}^{c_2\,t}}Nc_1\,c_2\,t
-2\,{N}^{2}c_1\,tc_2\,x_0\\
&-Nx_0\,tc_2^2
+3\,N{c_1}\,{{e}^{2\,c_2\,t}}
+{{e}^{2\,c_2\,t}}c_2
+2\,Nx_0\,{{e}^{c_2\,t}}c_2\\
&+N^2c_1^2t
-6\,{{e}^{c_2\,t}}Nc_1
-{{e}^{c_2\,t}}c_2
-2\,N{x_0}\,{c_2}+3\,N{c_1}) 
,
\end{split}
\end{equation}
and
\begin{equation}
\begin{split}
\text{Var}(S_{\text{CGT}}) = \frac {N{{e}^{-2\,c_2\,t}}}{ c_1c_2^2}
(& N{c_1}\,t{c_2}\,{{e}^{2\,{c_2}\,t}}+{
{e}^{c_2\,t}}c_2^2tN x_0-{{e}^{c_2\,t}}
Nc_1\,c_2\,t
-Nx_0\,tc_2^2
+N{c_1}\,{{e}^{2\,c_2\,t}}\\
&+{{e}^{2\,c_2\,t}}c_2
-2\,{{e}^{c_2\,t}}Nc_1
-{{e}^{c_2\,t}}{c_2}
+N{c_1}) 
.
\end{split}
\end{equation}
We observe that the variance of the GT estimator is $\mathcal{O}(N^3)$ while that of 
the CGT estimator is $\mathcal{O}(N^2)$, as $N \to \infty$. On the other hand, both 
estimators have $\mathcal{O}(t)$ variance as $t \to \infty$.
Hence, in the $N \to \infty$ limit, the RSD of the GT estimator is $\mathcal{O}(N^{1/2})$ and 
the RSD of the CGT estimator is $\mathcal{O}(1)$.
We can also conclude that in the $t \to \infty$ limit, the 
RSD is $\mathcal{O}(\sqrt{t})$ for both methods.

Secondly we consider the sensitivity $\frac{\partial}{\partial c_2}
\bfE(X^N(t, c))$. The weight process $Z^N(t,c)$ in this case can be written as 
\begin{equation}
Z^N(t,c) = \int_{(0,t]}\frac{1}{c_2}dR^N_2(s,c) - \int_{0}^{t} X^N(s,c) ds,
\end{equation}
and the analysis, while possible is more complicated. For simplicity, we choose
$c_1=0$, so the process now corresponds to a pure death process. 
In this case, the variances of GT and CGT estimators can be shown to be
\begin{equation}\label{equ:VAR1}
\begin{split}
\text{Var}(S_{\text{GT}}) = \frac{1}{c_2^2}&(e^{-2 c_2 t} N^3 x_0^3 - 4e^{-2c_2 t}
N^2 x_0^2+3e^{-2c_2 t} N x_0+3e^{-2 c_2 t} N^2 x_0^2 t^2 c_2^2\\
&-2e^{-3c_2 t} N x_0 + 3e^{-3c_2 t} N^2 x_0^2 + e^{-c_2 t} N^2 x_0^2 - e^{-c_2
  t} N x_0\\
&+e^{-c_2 t} N x_0 t^2 c_2^2-4e^{-2c_2t}t^2c_2^2 N x_0 - e^{-3c_2 t} N^2 x_0^3),
\end{split}
\end{equation}
 and
\begin{equation}\label{equ:VAR2}
\begin{split}
\text{Var}(S_{\text{CGT}})=\frac{1}{c_2^2}&(-2e^{-2c_2t} N^2 x_0^2
+ 3e^{-2c_2 t} N x_0 + e^{-2c_2 t} N^2 x_0^2 t^2 c_2^2 \\
&- 2e^{-3c_2t} N x_0 + e^{-3c_2t} N^2 x_0^2 + e^{-c_2t} N^2 x_0^2\\
&- e^{-c_2t} N x_0 + e^{-c_2t} N x_0 t^2 c_2^2 - 4e^{-2c_2t} N x_0 t^2 c_2^2).
\end{split}
\end{equation}
When dependence on system size $N$ is
concerned, the variance of GT estimator is $\mathcal{O}(N^3)$ while that of CGT
estimator is only $\mathcal{O}(N^2)$. 
As in the case of the parameter $c_1$, 
we again obtain that the RSD of the GT method is $\mathcal{O}(N^{1/2})$ while that of
CGT is $\mathcal{O}(1)$, as $N \to \infty$. 
Finally, we note that large $t$ behavior is uninteresting as the 
system enters the absorbing state $0$ eventually.

Now we consider any FD estimator, and we can bound its variance as 
\begin{equation}
\begin{split}
\text{Var}(S_{\text{FD}})&=
h^{-2}\text{Var}(X^N(t, c+h) - X^N(t, c)) \\
&\leq 2h^{-2}\left\{\text{Var}(X^N(t, c+h)) + \text{Var}(X^N(t, c))\right\}.
\end{split}
\end{equation}
We also note that \cite{rathinam2007reversible}
\begin{equation}
\text{Var}(X^N(t, c)) = N x_0 (1-e^{-c_2 t})e^{-c_2 t} + \frac{N c_1}{c_2} (1-e^{-c_2t}).
\end{equation}
In our analysis we shall treat the finite difference perturbation $h$ of the parameter
as independent of system size $N$ so that we consider the variance and bias of the FD estimator as a function of the two variables $h$ and $N$.
From the above equation, we see that for any fixed $h$, the variance of an FD estimator is $\mathcal{O}(N)$ and hence the RSD of the FD estimator is $\mathcal{O}(N^{-1/2})$ as $N \to \infty$.
Finally, we note that for fixed $N$, as $t\to \infty$, the variance of the FD estimator is $\mathcal{O}(1)$.

We note here that the above upper bound for $\text{Var}(S_{\text{FD}})$ is exactly twice the variance of the
independent random number (IRN) FD method. If a common random number (CRN) 
FD method is used, the variance is in general much smaller.  Nevertheless, our
numerical results show that the asymptotic order in $N$ is sharp even for CRN.  

From the expression for $\bfE(X^N(t,c))$ in \eqref{eq-EXN-bd} it can easily be
shown that the relative bias (RB) defined by \eqref{eq-RB-def}, of any FD method is $\mathcal{O}(1)$ as $N \to \infty$ (with $h$ fixed) when
sensitivity of $\bfE(X^N(t,c))$ with respect to $c_1$ or $c_2$ is
considered.

To summarize, we note that when computing the sensitivity of $\bfE(X^N(t,c))$ 
in this example, with respect to either of the parameters $c_1$ or $c_2$, we observe that the 
RSDs of the GT, CGT and FD estimators scale with system size
$N$ as $\mathcal{O}(N^{1/2}),\mathcal{O}(1)$ and $\mathcal{O}(N^{-1/2})$ respectively. If $N$ is modestly large (say $10-100$), a
significant amount of reduction in the RSD can be expected using CGT over GT. 
On the other hand FD methods will have even lower variance when compared 
to both GT and CGT as system size increases. However, the FD methods are biased, 
and for fixed $h$ the relative bias (RB) remains $\mathcal{O}(1)$ as $N \to \infty$. 

\subsection{Contributions of this paper} 
Our analysis will show that the observations made about the relative standard
deviation (RSD) and the relative bias (RB) of 
the GT, CGT and FD estimators in the context of the particular example of the
previous subsection generalize to a large class of stochastic reaction
networks.  These general results are provided in Section \ref{sec-biasvar-sens}.
While our analysis does not apply to the RPD method, our numerical simulations
show that RPD has system size dependence similar to the FD methods. 
While our RSD analysis in the cases of CGT and CRN FD estimators is 
not proven to be sharp, the numerical simulations show that the estimates 
in terms of large system size $N$ are sharp. 

Our analysis thus provides theoretical evidence that centering (to obtain the
CGT method) significantly
improves the efficiency of the GT methods. Since the FD methods are biased 
while the GT and CGT methods are not, efficiency comparison must be based on variance and
bias. In the case of the FD estimators which depend on system size $N$ as 
well as the perturbation parameter $h$, our analysis in Section
\ref{sec-biasvar-sens} treats $h$ and $N$ as 
independent variables and provides the large $N$ behavior for fixed $h$. 
The small $h$ behavior of the FD methods (for fixed $N$) is well 
known \cite{Glynn-book}.   
In Section \ref{sec-discuss}, we combine our large $N$ results with the 
existing small $h$ results for the FD methods in order to decide the optimal
choice of $h$ as a function of $N$, and  
provide an estimate of efficiency (as measured by the number $N_s$ of trajectories needed to achieve a 
given value $\delta$ for the relative error (RE)) of the GT, CGT and FD methods.

\section{General setup and running assumptions}
As mentioned in the previous section, the system size shall be the key to 
our analytical explanation for the larger variance of the GT estimator. 
In this section we set the stage for the system size analysis and state some
assumptions that shall be carried throughout the rest of the paper. 
We shall use the notation $|x|$ for the norm of a vector (any norm in
$\real^n$ would do) and $\|\nu\|$ for the corresponding induced
norm of a matrix.

\begin{remark}\label{rem-no-c}
Our analysis will focus on processes $X$, $R$ and $Z$ corresponding to
different system sizes $N$, however, the deterministic parameter value $c$ is fixed at 
a specific value at which the sensitivity is sought. For notational ease
and readability, we
shall not show the dependence of these processes and intensity functions 
on $c$, and only display $c$ when it 
explicitly appears outside these.      
\end{remark}
  
We will study the family of processes $X^N$ indexed
by $N \geq 1$ corresponding to the family of intensity functions $a_j^N$ 
that are represented on the same sample space via
the stochastic equation
\begin{equation}\label{RTC-N1}
X^N(t) = N x_0 + \sum_{j = 1}^m Y_j\left(\int_0^t a_j^N(X^N(s))\, ds\right) \nu_j, \quad N \geq 1,
\end{equation} 
where $Y_j$ are independent unit rate Poisson processes and we have 
taken $X^N(0) = N x_0$ where $x_0 \in \real_+^n$ is fixed (deterministic).
We also define the corresponding family of vector reaction count processes
$R^N(t)$ whose $j$th component $R^N_j(t)$ counts the number of reaction events
of type $j$ that occurred during $(0,t]$.  Thus 
\[
R^N_j(t) = Y_j\left(\int_0^t a_j^N(X^N(s))\, ds\right), \quad N \geq 1,\;\;
j=1,\dots,m.
\]  
We also define the centered processes $M^N(t) = (M^N_1(t),\dots,M^N_m(t))$ by
\[
M^N_j(t) = R^N_j(t) - \int_0^t a_j^N(X^N(s)) ds, \quad N \geq 1,\;\;
j=1,\dots,m.
\]

We shall state five running assumptions under which the rest of the
analysis in this paper is carried out. We note that the Assumptions 1-3 are 
assumptions on the intensity functions and their dependence on parameters 
and system size. These assumptions are satisfied by the stochastic mass 
action form of intensity functions and are intended to generalize certain 
key properties of the stochastic mass action form of intensity functions. 
Not all stochastic models of intensity functions in the literature follow
the stochastic mass action form. In such cases, our analysis will still apply 
provided these assumptions are met.

\begin{assume}\label{assume1}
We assume the following form of parameter dependence on 
the intensity function. For each $j = 1, \cdots, m$ and $N \geq 1$,
\begin{equation}
a_j^N(x, c) = c_j b_j^N(x),
\end{equation}
where $b_j^N:\real^n \to \real$ are such that $b_j^N$ restricted to
$\posint^N$ are nonnegative. This also implies that there are precisely $m$ parameters, one for each reaction $j$. 
\end{assume}

For the analysis in this paper we need not assume the stochastic mass action
form, but merely the density dependence which is stated by our Assumption \ref{assume2}. 

\begin{assume}\label{assume2}
We suppose that for each $j=1,\cdots,m$, and each $x \in \real_+^n$, 
the limit $\lim_{N \to \infty} a_j^N(Nx)/N = a_j(x)$ exists and moreover, 
for each compact $K \subset \real_+^n$, the collection of functions $a_j^N(N
x) -N a_j(x)$ is uniformly bounded for $x \in K$ and $N \geq 1$.  
We note that this implies that for each compact set $K \subset \real_+^n$ 
there exists a constant $B_K>0$ such that
\begin{equation}\label{eqn:assume2}
\left|\frac{a_j^N(N x)}{N}-a_j(x)\right| \leq \frac{B_K}{N}, \quad x \in K, \; j=1,\dots,m, \; N \geq 1.
\end{equation}
\end{assume}

Defining $X_N(t) = N^{-1} X^{N}(t)$, we note that $X_N$ can be interpreted as
the {\em concentration} of molecules at time $t$ for system size $N$.
We note that $X_N$ are coupled via the following stochastic equations.
\begin{equation}\label{RTC-N}
X_N(t) = x_0 + \sum_{j = 1}^m N^{-1} Y_j\left(\int_0^t a_j^N(NX_N(s))\, ds\right) \nu_j.  
\end{equation}

We state the following theorem regarding the limiting behavior of $X_N$ (see
\cite{Ethier-Kurtz} for details). The deterministic limit $X$ of $X_N$ is also
referred to as the {\em fluid limit}.

\begin{theorem}\label{thm:fluid}(Theorem $2.1$ of Chapter $11$ in \cite{Ethier-Kurtz})
Suppose that Assumption \ref{assume2} holds. Moreover,  assume that for each compact $K \subset \mathbb{R}^n$,
\[\sum_{j=1}^{m} |\nu_j|\sup_{x\in K} a_j(x) < \infty,\]
and that $F(x) = \sum_{j = 1}^{m} \nu_j a_j(x)$ is Lipschitz on $K$, that is, for each $x, y\in K$, there exists some constant $M_K$ such that
\[|F(x) - F(y)| \leq M_K|x - y|.\]
Suppose $t >0$ is in the forward maximal interval of existence of solution $X$
for the ODE initial value problem  
\[
X(t) = x_0 + \int_0^t F(X(s)) ds.
\]
Then
\[\lim_{N}\sup_{s \leq t} \left|X_N(s) - X(s)\right| = 0~~\text{a.s.}, \]
where the deterministic limit $X$ satisfies the ODE above.
\end{theorem}

\begin{remark} 
We note that with fixed initial condition $X_N(0)=x_0$ we  
want $X^N(0)= N x_0$ to belong to $\mathbb{Z}_+^n$, which may not hold for all $N
\geq 1$ but we assume that it holds for a sequence of $N$ values tending to
$\infty$. For instance if $x_0$ is rational this is true. 
This is adequate for our purposes. 
\end{remark}

In order to satisfy the conditions stated in Theorem \ref{thm:fluid} we shall 
assume the following. 

\begin{assume}\label{assume3}
For each $j=1,\dots,m$, the functions $a_j(x):\real^n \to
\real$ are continuously differentiable. This automatically implies the
Lipschitz condition in Theorem \ref{thm:fluid}. 
\end{assume}

The following assumption is used to facilitate the analysis in this paper. 
Several, but not all examples in applications satisfy this assumption. 
  
\begin{assume}\label{assume4}
We assume that the sequence of concentration processes $X_N$ is uniformly bounded, that is, there exists a constant $\Gamma$ such that
for all $t \geq 0$,
\begin{equation}\label{eqn:assume4}
|{X_N}(t)| \leq \Gamma ~~~\text{a.s.} 
\end{equation}
for all $N \geq 1$.
\end{assume}

We note that if there exists a strictly positive vector $\gamma \in \real_+^m$ so that 
$\gamma^T \nu_j \leq 0$ for each $j$ then this assumption is satisfied. 
We note that a form of converse of this statement is also true \cite{Rathinam-QAM}.

Now we turn our attention to the sensitivity. Given $f:\real^n \to \real$,
we are interested in computing the sensitivity 
\[
\frac{\partial}{\partial c} \mathbb{E}(f(X^N(t))),
\]
where $c \in (0,\infty)$ is a parameter. In view of Assumption \ref{assume1}, without loss
of generality, we shall take $c=c_1$. Then we note that the GT sensitivity
estimator is $f(X^N(t))Z^N(t)$ and the CGT estimator is
$[f(X^N(t))-\mathbb{E}(f(X^N(t))]Z^N(t)$, where we note that $Z^N(t) =
M^N_1(t)/c_1$ in this case. 

As we are concerned with families of processes indexed by $N$, it makes sense to consider a corresponding family of functions $f^N:\real^n \to \real$ instead of one function $f$ and make reasonable assumptions on $f^N$ and $f$.  

To motivate the assumption we make on $f^N$ and $f$, we note that we shall be concerned with $f^N(X^N(t)) = f^N(N X_N(t))$ which we wish to compare with $f(X(t))$. 
When $f^N(x) = x_i$, one of the components of $x$, we have 
\[
f^N(N X_N(t))/N =  {X_N}_i(t) \to X_i(t) = f(X(t)),
\]
with $f(x)=x_i$. 
Alternatively, if $f^N(x) = x_i^\alpha$ for some $\alpha>0$ we have 
\[
f^N(N X_N(t))/N^\alpha = ({X_N}_i(t))^\alpha \to (X_i(t))^\alpha  = f(X(t)),
\]
with $f(x)=x_i^\alpha$. If however $f^N(x) = x_i^2 + x_i$ then we have
\[
f^N(N X_N(t))/N^2 = ({X_N}_i(t))^2 + {X_N}_i(t)/N \to (X_i(t))^2 = f(X(t)),
\]
where $f(x) = x_i^2$. In this case we note that $f^N(Nx)/N^2 - f(x) = x_i/N$ 
which tends to $0$ as $1/N$, uniformly for $x$ in a compact set.  
Motivated by this, we impose the following assumption. 

\begin{assume}\label{assume5}
We assume that there exist a function $f$ and a constant $\alpha > 0$ such that
for each compact set $K \subset \real_+^n$, 
\begin{equation}
\left| f^N(Nx)/N^\alpha  - f(x) \right| \leq \frac{L_K}{\sqrt{N}}, \quad x \in K, \;\; N \geq 1,
\label{eq-fN-limit}
\end{equation}
for some constant $L_K > 0$.
\end{assume}

We remark that the $\mathcal{O}(1/\sqrt{N})$ behavior is adequate for our proofs.

We note that the running assumptions 1-5 will be assumed throughout the rest of the paper.

\section{Large $N$ behavior}

In this section we derive results concerning the $N \to \infty$ limit 
for the various relevant processes.
Throughout the rest of the paper $X(t)$ will denote the solution of the 
equation
\begin{equation}\label{eq-X}
X(t) = x_0 + \sum_{j=1}^m \nu_j \int_0^t a_j(X(s)) ds,
\end{equation}
where $x_0 \in \real_+^n$ is fixed. 

\begin{lemma}\label{lem-Aj}
For each $j = 1, \cdots, m$, there exists
$A_j > 0$ such that, for all $t > 0$
\[\frac{a_j^N(N X_N(t))}{N} \leq A_j~~~\text{a.s.}\]
for all $N \geq 1$.
\end{lemma}
\begin{proof}
By Assumption \ref{assume4}, the processes $X_N$ are contained in a compact set of $\mathbb{R}^{n}$, say $K$, therefore, for each $j$ we
have the estimation
\begin{equation*}
\sup_{t\geq 0} \frac{a_j^N(N X_N(t))}{N} \leq \sup_{x \in K} \frac{a_j^N(N x)}{N}.
\end{equation*}
Since $N^{-1}a_j^N(N x)$ converges uniformly to $a_j(x)$ for $x \in K$ by \eqref{eqn:assume2} in Assumption \ref{assume2}, it is apparent that
$\sup_{x \in K} N^{-1}{a_j^N(N x)}$ is bounded by continuity of $a_j$. Hence $\sup_{t \geq 0} N^{-1}{a_j^N(N X_N(t))}$ is bounded by a constant $A_j$.
\end{proof}

\begin{lemma}\label{lem:intensity}
For each
$j = 1, \cdots, m$, and $t > 0$, we have 
\[
\sup_{s \leq t} \left|\frac{a_j^N(NX_N(s))}{N}-a_j(X(s))\right| \to 0, \quad \text{a.s.}
\]
as $N \to \infty$.
\end{lemma}
\begin{proof}
We may write 
\begin{equation*}
\begin{split}
&\left|\frac{a_j^N(NX_N(s))}{N} - a_j(X(s))\right| \\
 \leq & \left|\frac{a_j^N(NX_N(s))}{N} - a_j(X_N(s))\right| + \left|a_j(X_N(s)) - a_j(X(s))\right|.
\end{split}
\end{equation*}
The first part on the right hand side converges to zero uniformly for $s $ in $ [0,t]$ because of Assumption \ref{assume2} and Assumption \ref{assume4}.
To see that the second part on the right hand side converges uniformly to $0$ on $[0, t]$, note that by Assumption \ref{assume3} and Assumption \ref{assume4}, $a_j$ is Lipschitz continuous on the compact set $K$ (which contains $X_N$ and $X$), hence the result follows by Theorem
\ref{thm:fluid}.   
\end{proof}

We define a family of scaled reaction count processes $R_N(t)$ by $R_N(t)=R^N(t)/N$.
\begin{lemma}\label{thm:R-fluid}
For each $j = 1, 2,\cdots, m$ and $t > 0$,
\[
\sup_{s \leq t} \left|{R_N}_j(s) - \int_0^s a_j(X(u)) du \right| \to 0 ~~\text{a.s.}\]
as $N \to \infty$.
\end{lemma}
\begin{proof}
Recall that $R_j^N(t) = Y_j\left(\int_0^t a_j^N(NX_N(s)) ds\right)$. For each $j = 1, \cdots, m,$
\begin{equation*}
\begin{split}
&\sup_{s \leq t} \left|\frac{1}{N}Y_j\left(\int_0^s a_j^N(NX_N(u)) du\right) - \int_0^s a_j(X(u)) du \right|\\
\leq & \sup_{s \leq t} \left|\frac{1}{N}Y_j\left(\int_0^s a_j^N(NX_N(u)) du\right) - \frac{1}{N}\int_0^s a_j^N(NX_N(u)) du\right|\\
  & +  \int_0^t \left| \frac{1}{N}a_j^N(NX_N(u)) - a_j(X(u))
\right| du.
\end{split}
\end{equation*}
The second term on the right hand side converges to zero 
by Lemma \ref{lem:intensity}. 
Setting $\tilde{Y}(t) = Y(t) - t$,
the first term on the right can be written and then bounded as
\begin{equation*}
\sup_{s \leq t} \left|\frac{1}{N}\tilde{Y}_j\left(\int_0^s a_j^N(NX_N(u)) du\right)\right|
 \leq \sup_{s \leq t} \left|\frac{1}{N}\tilde{Y}_j\left(NA_j s\right)\right| \quad \text{a.s.},
\end{equation*}
where the last term converges to zero by the law of large numbers for Poisson processes (see Theorem $1.2$ in \cite{Kurtz-survey} ). 
\end{proof}

\begin{lemma}\label{lem-fN-limit}
For a given $t > 0$, suppose that $f$ is continuous at $X(t)$.
Then
\begin{equation}
\lim_{N \to \infty} |f^N(N X_N(t))/N^\alpha -f(X(t))| = 0, \;\; \text{a.s.}
\end{equation}
\end{lemma}
\begin{proof}
Write
\begin{equation*}
\begin{split}
\left| f^N(N X_N(t))/N^\alpha - f(X(t)) \right| \leq  &\left|  f^N(N X_N(t))/N^\alpha - f(X_N(t))\right| \\
& + \left| f(X_N(t))-f(X(t))\right|.
\end{split}
\end{equation*}
The first term converges to zero almost surely by Assumption \ref{assume4} and \eqref{eq-fN-limit} in Assumption \ref{assume5}. The second term
converges to zero by the continuity assumption on $f$ since $X_N(t)$ converges to $ X(t)$ almost surely.
\end{proof}
%

Recall the definition of $M^N$,
\[M^N(t) = R^N(t) - \int_0^t a^N(N X_N(s)) ds.\]
Note that in general, $M^N(t)$ is an $m$-dimensional local martingale (see \cite{Protter, Jacod-Shiryaev} for definition) for each $N$, but by Lemma
\ref{lem-Aj} it follows that $\mathbb{E}[R_j^N(t)] \leq N A_j t$ for all $t > 0$
which makes $M^N(t)$ a martingale. 
We define the scaled processes $M_N = N^{-1}M^N$ and $Z_N = N^{-1}Z^N$. 
We note that $Z^N(t) = M^N_1(t)/c_1$ and $Z_N(t) = {M_N}_1(t)/c_1$.

Let us denote by $D^m[0, \infty)$ the space of c\`adl\`ag functions mapping
  from $[0,\infty)$ into $\mathbb{R}^m$, endowed with the Skorohod topology (see
    \cite{pb1999} for definitions).
We provide a lemma on the weak convergence of $M_N$.
\begin{lemma}\label{thm:M-CLT}
Let $C(t) = (c_{ij}(t))$ be the $m \times m$ matrix-valued function, where
\begin{equation}
c_{ij}(t) = \left\{
\begin{array}{cc}
\int_0^t a_j(X(s)) ds  &i = j\\
0 & i \neq j.
\end{array}\right.
\end{equation}
Then $\sqrt{N}M_N \Rightarrow \bar{M}$ on $D^m[0, \infty)$,
where $\bar{M}(t)$ is an $m$-dimensional Gaussian process with independent increments, having mean vector and covariance matrix
\begin{equation}
\mathbb{E}[\bar{M}(t)] = (0, \cdots, 0), \;\; \mathbb{E}[\bar{M}(t)\bar{M}(t)^{T}] = C(t).
\end{equation}
In particular, the scaled Girsanov sensitivity (or weight) process $\sqrt{N}Z_N \Rightarrow U$ on $D[0, \infty)$, where
\begin{equation}\label{equ:U}
U(t) = \frac{1}{c_1} \bar{M}_1(t).
\end{equation}
Also since $U$ has continuous sample paths, for each $t > 0$, we have
\[\sqrt{N}Z_N(t) \Rightarrow U(t).\] 
\end{lemma}

\begin{proof}
The proof relies on the martingale functional central limit theorem (FCLT) proved in \cite{Whitt}. Note that each jump of $\sqrt{N}M_N$ has size 
$1/{\sqrt{N}}$, therefore,
\[
\lim_{N\to \infty} \mathbb{E}\left[\sup_{s \leq t}\left|\sqrt{N}M_N(s) - \sqrt{N}M_N(s-)\right|\right] = 0.
\]
Also, for each pair $(i,j)$ with $i, j = 1, \cdots, m$, and each $t > 0$, since the jump size for ${M_N}_j$ is always $N^{-1}$ and there are no
simultaneous jumps, we have the following quadratic covariation
\begin{equation}
\left[\sqrt{N}{M_N}_i, \sqrt{N}{M_N}_j\right](t) = \left\{
\begin{array}{cc}
{R_N}_j  &i = j\\
0 & i \neq j.
\end{array}\right.
\end{equation}
By Lemma \ref{thm:R-fluid}, ${R_N}_j(t)$ converges almost surely to $c_{jj}(t) = \int_0^t a_j(X(s))ds$. Then,
for each pair $(i,j)$,
\[\left[\sqrt{N}{M_N}_i, \sqrt{N}{M_N}_j\right](t) \to c_{ij}(t)\]
almost surely and hence in probability. Thus, the weak convergence of $M_N$
follows from the martingale FCLT.
\end{proof}



\begin{lemma}\label{lem:moment-M}
For each $p \geq 1$,
there exists a constant $\beta(p)$ such that for all $t > 0$
\begin{equation}
\limsup_N \mathbb{E}\left(\sup_{s \leq t}\left|\sqrt{N}
M_{N}(s)\right|\right)^p \leq \beta(p) t^{p/2}.
\end{equation}
\end{lemma}

\begin{proof}
Observe that the quadratic variation (see \cite{Protter} for definition) of $\sqrt{N} M_{N}$ is

\[\left[\sqrt{N} M_{N}, \sqrt{N} M_{N}\right](t)   =   N^{-1}  \sum_{j=1}^{m}  Y_j\left(\int_0^t a_j^N (N X_N(s))ds\right).\]

By the Burkholder-Davis-Gundy inequality (see \cite{Protter}), there exists a constant $C(p)$ (depends on $p$) such that
\begin{equation*}
\begin{split}
\mathbb{E}\left(\sup_{s \leq t}\left|\sqrt{N}  M_{N}(s) \right|      \right)^p \leq  &C(p) \mathbb{E}\left( \frac{1}{N} \sum_{j=1}^m   Y_j\left(\int_0^t a_j^N(N X_N(s))ds\right)\right)^{p/2}\\
\leq & C(p)\mathbb{E} \left( \frac{1}{N} \sum_{j=1}^m  Y_j\left(NA_jt \right)\right)^{p/2}\\
\leq & C(p) N^{-p/2} \left(\mathbb{E}\left(\sum_{j=1}^m  Y_j(N A_j t)   \right)^p\right)^{1/2},
\end{split}
\end{equation*}
where we have used Lemma \ref{lem-Aj}. 

Hence,
\[\limsup_{N}  \mathbb{E}\left(\sup_{s \leq t}\left|\sqrt{N} M_{N}(s) \right|      \right)^p  \leq    \limsup_N C(p) N^{-p/2} \left(\mathbb{E}\left(\sum_{j=1}^m Y_j(N A_j t)   \right)^p\right)^{1/2}.\]

First we observe that for $j=1,\cdots, m$, the $p$th moment of the Poisson
random variable $Y_j(NA_j t)$ is a polynomial of degree $p$ in $NA_j t$. Also,
noting that $Y_j$ are independent, we obtain that the right hand side is
bounded by a term $\beta(p)t^{p/2}$, where $\beta(p)$ is a constant.

\end{proof}

Since $Z^N(t) = {c_1}^{-1}M_1^N(t)$, we immediately have the following property regarding the process $Z_N$.
\begin{lemma}\label{lem:moment-Z}
For each $p \geq 1$, there exists a constant $\gamma(p)$ such that for all $t > 0$,
\begin{equation}
\limsup_N \mathbb{E}\left(\sup_{s \leq t}\sqrt{N} \left|
Z_N(s)\right|\right)^p \leq \gamma(p)t^{p/2}.
\end{equation}
\end{lemma}

Define the process $V_N(t) = \sqrt{N}(X_N(t) - X(t))$. Let us consider the moment of this process on a compact time interval.

\begin{lemma}\label{lem:moment-V}
For each $p \geq 1$, there exist constants $\bar{\beta}(p), K(p)$ such that for all $t > 0$
\[\limsup_N \sup_{s\leq t} \mathbb{E}\left(|V_{N}(s)|^p\right) \leq
\bar{\beta}(p) t^{p/2} e^{K(p) t^p}.  \]
\end{lemma}

\begin{proof}
Recall that
\[X_N(s) = x_0 + \nu R_N(s)\]
and
\[X(s) = x_0 + \int_0^s \nu a(X(u)) du,\]
where $\nu$ is the $n$ by $m$ dimensional stoichiometric matrix.
One can write $V_{N}$ as
\begin{equation*}
\begin{split}
V_{N}(s) = &\sqrt{N} \nu R_N(s) - \sqrt{N} \int_0^s \nu a(X(u)) du \\
            = &\sqrt{N} \nu \left( R_{N}(s) - \int_0^s \frac{a^N(NX_N(u))}{N} du\right)\\
               & + \sqrt{N} \nu \left(\int_0^s \frac{a^N(NX_N(u))}{N} -  a(X(u))  du\right).
\end{split}
\end{equation*}
Note that we denote  $M_N(s) = R_N(s) - \int_0^s N^{-1}a^N(NX_N(u))  du$, and hence
\begin{equation*}
\begin{split}
\left|V_{N}(s)\right|
        \leq & \|\nu\| \left| \sqrt{N}M_{N}(s)\right| +  \|\nu\| \int_0^s \sqrt{N}  \left|  \frac{a^N(NX_N(u))}{N} -  a(X(u))\right| du.
\end{split}
\end{equation*}
To estimate the second term on the right hand side of the last inequality,
we note that
\begin{equation*}
\begin{split}
\sqrt{N} \left|\frac{a^N(NX_N(u))}{N} - a(X(u))\right| \leq  & \sqrt{N} \left|\frac{a^N(NX_N(u))}{N} - a(X_N(u))\right|\\
                                     & + \sqrt{N} \left|a(X_N(u)) - a(X(u))\right|.\\
\end{split}
\end{equation*}
Since $X_N$ lies in a compact set $K$ according to Assumption \ref{assume4}, we have for all $u > 0 $,
\[\left|\frac{a^N(NX_N(u))}{N} - a(X_N(u))\right| \leq \frac{\tilde{B}_K}{N},\]
where we have used Assumption \ref{assume2} and 
$\tilde{B}_K$ is related to $B_K$ from \eqref{eqn:assume2}.


On the other hand, for each $j = 1, \cdots, m$, by Assumption \ref{assume3}, $a_j$ is continuously differentiable and hence it is Lipschitz continuous on the compact set $K$. Hence, there exists a Lipschitz constant $C_j$ such that for all $u>0$,
\[\left|a_j(X_N(u)) - a_j(X(u))\right| \leq C_j\left|X_N(u) - X(u)\right|.\]
It follows that there exists a constant $C$ such that
\[\left| a(X_N(u)) - a(X(u))\right| \leq C\left|X_N(u) - X(u)\right|,\]
where $|\dot|$ can be norm on $\mathbb{R}^m$.
Therefore, 
\begin{equation*}
\begin{split}
\left|V_{N}(s)\right|
        \leq &\|\nu\|\left( \left| \sqrt{N} M_{N}(s)\right| + N^{-1/2}\tilde{B}_K s 
         +  C \int_0^s  \sqrt{N}\left| X_N(u) - X(u) \right|   du  \right)          \\
        = &\|\nu\| \left(\left| \sqrt{N} M_{N}(s)\right| + N^{-1/2}\tilde{B}_K s  + C\int_0^s \left|V_{N}(u)\right| du \right).
\end{split}
\end{equation*}

In virtue of the inequality $(a+b+c)^p \leq 3^p(a^p + b^p + c^p) $ and the Holder's inequality,
we obtain
\begin{equation*}
\begin{split}
|V_{N}(s)|^p \leq &(3\|\nu\|)^p \left(\left| \sqrt{N} M_{N}(s)\right|^p + N^{-p/2}(\tilde{B}_K s)^p + C^p s^{p-1}\int_0^s |V_{N}(u)|^p du \right) .
\end{split}
\end{equation*}
Taking expected value of both sides, for $s \in [0, t]$,  
\begin{equation*}
\begin{split}
\mathbb{E}|V_{N}(s)|^p \leq & (3\|\nu\|)^p \left(\mathbb{E}\left| \sqrt{N} M_{N}(s)\right|^p +  N^{-p/2}(\tilde{B}_K t)^p \right)\\
 &+(3\|\nu\|)^p  C^p s^{p-1}\left(\int_0^s \mathbb{E}|V_{N}(u)|^p du\right).\\
\end{split}
\end{equation*}
To estimate the first term of the right hand side,  recall that in the proof of Lemma \ref{lem:moment-M}, 
\[\mathbb{E}\left(\sup_{s \leq t} \left| \sqrt{N} M_{N}(s)\right|\right)^p \leq C(p) N^{-p/2} \left(\mathbb{E}\left(\sum_{j=1}^m  Y_j(N A_j t)   \right)^p\right)^{1/2}.\]
For convenience, let us denote 
\[\Phi_N(t) = C(p) N^{-p/2} \left(\mathbb{E}\left(\sum_{j=1}^m  Y_j(N A_j t)   \right)^p\right)^{1/2}.\]
Therefore, 
\begin{equation*}
\begin{split}
\mathbb{E}|V_{N}(s)|^p \leq & (3\|\nu\|)^p \left(\Phi_N(t) +  N^{-p/2}(\tilde{B}_K t)^p 
+  C^p s^{p-1}\left(\int_0^s \mathbb{E}|V_{N}(u)|^p du\right) \right).
\end{split}
\end{equation*}
We note that $\mathbb{E}|V_{N}(s)|^p$ is continuous in $s$ and 
applying the Gronwall inequality, we obtain that, for $s \leq t$, 

\begin{equation*}
\begin{split}
\mathbb{E}|V_{N}(s)|^p \leq    (3\|\nu\|)^p \left(\Phi_N(t)
 +   N^{-p/2}(\tilde{B}_K t)^p\right) e^{(3\|\nu\|)^p C^p t^p}.
\end{split}
\end{equation*}
Taking supremum over $s \in [0, t]$ and then taking $\limsup_N$,
the result follows from same considerations as in the proof of Lemma \ref{lem:moment-M}.
\end{proof}

%
\section{Scaling of sensitivity, estimator bias and estimator variance}\label{sec-biasvar-sens}

In this section, we study the system size dependence of the sensitivity
\[\mathscr{S}^N=\frac{\partial}{\partial c}\mathbb{E}(f^N(X^N(t))),\]
and the biases as well as the variances of the GT, CGT and FD estimators.
In the case of the FD estimators, the parameter perturbation $h$ is fixed 
when $N \to \infty$. 
As mentioned earlier, the difference between the sensitivity with respect to
the stochastic parameter and with respect to the deterministic parameter is
merely a scaling factor $N^{|\nu''_j|-1}$ and hence the RSD, RB and RE are unchanged regardless of whether one considers the sensitivity with respect to the stochastic parameter or the deterministic parameter. From an analytical point of view, it is convenient to study the sensitivity with respect to the deterministic parameter. 

Recall that the sensitivity estimator of the Girsanov transformation method is
\[f^N(X^N(t, c)) Z^N(t, c)\]
where $f^N: \mathbb{R}^n \to \mathbb{R}$. 
We remind the reader that $f^N$ satisfies the Assumption $5$, that is, there exist a function $f$ and a constant $\alpha$ such that 
\[\left|\frac{f^N(Nx)}{N^{\alpha}}-f(x)\right| \leq \frac{L_K}{\sqrt{N}}.\]  

\begin{theorem}\label{thm:sens-rate}
In addition to our running assumptions, we assume that $f$ in \eqref{eq-fN-limit} is continuously differentiable.
Then for each $t\geq 0$
\[\sup_{s \leq t}\mathbb{E}(f^N(X^N(s)) Z^N(s)) = \mathcal{O}(N^{\alpha}).\]
That is, the true sensitivity is asymptotically $\mathcal{O}(N^{\alpha})$ uniformly on $[0, t]$.  
\end{theorem}

\begin{proof}
It is sufficient to show that $\sup_{s\leq t}\mathbb{E}(f^N(X^N(s)) Z^N(s))/N^{\alpha}$ is bounded in $N$.
Instead of working with $\mathbb{E}(f^N(X^N(s)) Z^N(s))/N^{\alpha}$, we use
\[\mathbb{E}\left(\frac{f^N(X^N(s))}{N^{\alpha}}Z^N(s) - f(X(s))Z^N(s)\right)\]
because they are equal but the latter is easier to work with. 

Note that $f$ is continuously differentiable hence Lipschitz on the compact set $K$ corresponding to Assumption \ref{assume4}. 
Denote by $C_K$ the Lipschitz constant for $f$. 
Using the assumptions on $f^N$ and $f$  
and writing $X^N$ in terms of $V_N$ as
\[X^N(s) = NX(s) + \sqrt{N}{V_N(s)},\]
which leads to 
\begin{equation*}\label{equ:centered-est}
\begin{split}
&\left|\frac{f^N(NX(s) + \sqrt{N}V_N(s))}{N^{\alpha}} - f(X(s))\right| |Z^N(s)| \\
\leq & \left|\frac{f^N(NX(s) + \sqrt{N}V_N(s))}{N^{\alpha}} - f\left(X(s)+\frac{V_N(s)}{\sqrt{N}})\right)\right| |Z^N(s)|\\
&+ \left|f\left(X(s)+\frac{V_N(s)}{\sqrt{N}}\right) - f(X(s))\right| |Z^N(s)|\\
\leq & \frac{L_K}{\sqrt{N}}|Z^N(s)| + C_K |V_N(s)| \frac{|Z^N(s)|}{\sqrt{N}}\\
\leq & L_K\sqrt{N}|Z_N(s)| + \frac{1}{2}C_K \left(|V_N(s)|^2 + N |Z_N(s)|^2\right),
\end{split}
\end{equation*}
where $L_K$ is as defined in Assumption 5. 
Taking expectation on both sides, the result follows from Lemmas \ref{lem:moment-Z} and \ref{lem:moment-V}.
\end{proof}

\begin{remark}\label{rem-sens-limit}
While the proof above does not show that the order $\mathcal{O}(N^\alpha)$ is sharp, 
it can be shown to be sharp, if under the $N^{\alpha}$ scaling, the
sensitivity of the stochastic process is shown to limit to the sensitivity of the deterministic limit $f(X(t))$ as $N \to \infty$. 
In fact, under additional assumptions, this limit can be shown \cite{Gupta-personal}.
Our numerical results in Section \ref{sec:numerical} also show $\mathcal{O}(N^\alpha)$ behavior. 
\end{remark}

Recall that the FD estimator is defined in \eqref{eqn:FD-estimator} as
\[S_{\text{FD}}^N(t,c) = h^{-1}[f^N(X^N(t,c+h)) - f^N(X^N(t,c))].\]
Based on the last theorem, with a little more effort we conclude the following corollary regarding the bias of FD estimator.
\begin{corollary}\label{cor:bias-FD}
In addition to the running assumptions, if we assume that $f$ is continuously differentiable,
then for each $t >0 $, we have
\[\mathbb{E}(S_{\text{FD}}^N(t) - \mathscr{S}^N(t)) = \mathcal{O}(N^{\alpha}),\]
where $\mathscr{S}^N(t)$ represents the true sensitivity at $t$.
That is, the bias of FD estimator is asymptotically $\mathcal{O}(N^{\alpha})$.
\end{corollary}
\begin{proof}
Since we have shown that the true sensitivity scales like $\mathcal{O}(N^{\alpha})$, it suffices to show that 
$\mathbb{E}(f^N(X^N(t,c)))$ is asymptotically of order $\mathcal{O}(N^{\alpha})$ for any $c$.
In fact, by Lemma \ref{lem-fN-limit}, 
$f^N(X^N(t))/N^{\alpha}$ converges almost surely to $f(X(t))$. To apply the dominate convergence theorem, note that the Assumption \ref{assume5} implies
\[\frac{|f^N(X^N(t))|}{N^{\alpha}} \leq |f(X_N(t))| + \frac{L_K}{\sqrt{N}}.\]
By virtue of the Assumption \ref{assume4}, the right hand side of the above equality is bounded in $N$ and hence it is integrable. Finally, the dominate convergence theorem gives the result. 
\end{proof}

Next, we investigate the variance of the GT estimator in terms of the system size $N$. The following lemma concerning the weak convergence of joint distribution is crucial for the proof of Theorem \ref{thm:GIR}.
\begin{lemma}\label{lem:joint-dist-simple}
 Let $X_n$ and $Y_n$ be $\real^m$ valued and $\real^k$ valued sequences
 of random variables, respectively.  Suppose $X_n$ converges to
 $X$ in probability (where $X$ is deterministic) and $Y_n \Rightarrow Y$. 
Then $(X_n, Y_n) \Rightarrow (X,Y)$ in $\real^{m+k}$.
\end{lemma}

\begin{proof}
Let $x \in \real^m$ be such that $X=x$ almost surely. 
First we show that $(X,Y_n) \Rightarrow (X,Y)$. If $f:\real^{m+k} \to \real$
is bounded and continuous then so is $g:\real^k \to \real$ defined by
$g(y)=f(x,y)$. Since $Y_n \Rightarrow Y$ we have that
\[
\mathbb{E}(f(X,Y_n))=\mathbb{E}(g(Y_n)) \to
\mathbb{E}(g(Y))=\mathbb{E}(f(X,Y)).
\]
Now $\|(X_n,Y_n)-(X,Y_n)\|=\|X_n-X\|$ and since $X_n \to X$ in probability,
$\|X_n-X\| \to 0$ in probability (implies convergence in distribution). Thus by 
Theorem $3.1$ in \cite{pb1999} we have that
$(X_n,Y_n) \Rightarrow (X,Y)$.
\end{proof}

\begin{theorem}\label{thm:GIR}
In addition to our running assumptions, we assume that $f$ in \eqref{eq-fN-limit} is bounded on every compact set and for a given $t > 0$, f is continuous at $X(t)$. Then we have,
\begin{equation}\label{equ:main_thm}
N^{-1-2\alpha} \mathbb{E} \left\{(f^N(X^N(t)))^2 (Z^N(t))^2 \right\} \to (f(X(t)))^2 \frac{1}{c_1} \int_0^t a_1(X(s)) ds
\end{equation}
as $N \to \infty$.
Furthermore, for each $t > 0$,
\[\sup_{s\leq t}\mathbb{E}\left((f^N(X^N(s)))Z^N(s)\right)^2 = \mathcal{O}(N^{2\alpha + 1}).\]
\end{theorem}
\begin{proof}
Lemma \ref{lem:moment-Z} implies the uniformly integrability of
$N^{-1} (Z^N(t))^2$.  By Assumption \ref{assume4} and
\eqref{eq-fN-limit} we have that 
$(f^N(X^N(t)))^2/N^{2 \alpha}$ is a uniformly bounded sequence.
Thus  $N^{-1-2 \alpha} (f^N(X^N(t)))^2 (Z^N(t))^2$ is uniformly integrable.

By Lemma \ref{lem-fN-limit} we have that $N^{-2 \alpha} (f^N(X^N(t)))^2$
converges to $(f(X(t)))^2$ almost surely. We also have that $N^{-1}Z^N(t)$
converge weakly to $U(t)$. Thus by Lemma \ref{lem:joint-dist-simple} and the continuous mapping theorem we have that 
\[
N^{-1-2 \alpha} (f^N(X^N(t)))^2 (Z^N(t))^2 \Rightarrow (f(X(t)))^2 U^2(t).
\]
By Theorem 3.5 from \cite{pb1999}, we note that if a uniformly integrable sequence 
converges weakly then it converges in the mean, hence the result \eqref{equ:main_thm} follows. 

Also, recall that $(f^N(X^N(t)))^2/N^{2 \alpha}$ is uniformly bounded, hence
\[N^{-2\alpha - 1}\sup_{s\leq t}\mathbb{E}\left((f^N(X^N(s)))Z^N(s)\right)^2 \leq \tilde{C}  \mathbb{E}(\sup_{s\leq t}\sqrt{N}|Z_N(s)|)^2.\]
Taking $\limsup_N$ and applying Lemma \ref{lem:moment-Z} yields the second result.
\end{proof}

Note that the above theorem does not assume $f$ is continuously differentiable. However, to state the result regarding the estimator variance for GT method, we still need to assume continuous differentiability on $f$ so that we can use Theorem \ref{thm:sens-rate}. 
\begin{corollary}\label{cor:rate-GT}
In addition to our running assumptions, we assume that $f$ in \eqref{eq-fN-limit} is continuously differentiable. Then for given $t > 0$,
the estimator variance of GT method is asymptotically $\mathcal{O}(N^{2\alpha+1})$ uniformly on $[0,t]$. 
\end{corollary}

%
Next, we will explore the variance of the centered Girsanov transformation approach.

\begin{theorem}\label{thm:rate-CGT}
In addition to our running assumptions, we assume that $f$ in \eqref{eq-fN-limit} is continuously differentiable.
Then for each $t > 0$,
\[\sup_{s\leq t}\mathbb{E}\left[(f^N(X^N(s)) - \mathbb{E}[f^N(X^N(s))] )Z^N(s)\right]^2 = \mathcal{O}(N^{2\alpha}).\]
\end{theorem}

\begin{proof} 
Write
\begin{equation*}
\begin{split}
&\mathbb{E}\left(  \left|\frac{f^N(X^N(s))}{N^{\alpha}} - \mathbb{E}\left(\frac{f^N(X^N(s))}{N^{\alpha}}\right)\right|^2 (Z^N(s))^2 \right)\\
\leq & 2\mathbb{E} \left(  \left|\frac{f^N(X^N(s))}{N^{\alpha}} - f(X(s))\right|^2 (Z^N(s))^2 \right) \\
 & +   2\mathbb{E}\left(  \left|f(X(s)) - \mathbb{E}\left(\frac{f^N(X^N(s))}{N^{\alpha}}\right)\right|^2 (Z^N(s))^2 \right)\\
\leq & 2\mathbb{E} \left(  \left|\frac{f^N(X^N(s))}{N^{\alpha}} - f(X(s))\right|^2 (Z^N(s))^2 \right) \\ 
& +  2\mathbb{E}\left(\left|\frac{f^N(X^N(s))}{N^{\alpha}} - f(X(s))\right|^2\right) \mathbb{E}(Z^N(s))^2,
\end{split}
\end{equation*}
where the last inequality is true due to the fact that $f(X(s))$ is deterministic.
Using similar argument as in the proof of Theorem \ref{thm:sens-rate},
the first term on the right-hand side can be bounded by 
\[4 L_K^2 \mathbb{E}\left({|\sqrt{N}Z_N(s)|}\right)^2 + 4 C_K^2 \mathbb{E} \left(|V_N(s)|\sqrt{N}{|Z_N(s)|}\right)^2.\]
Similarly, the second term on the right hand side can be bounded by
\[4 L_K^2 \mathbb{E}\left(\sqrt{N}|Z_N(s)|\right)^2 + 4 C_K^2 \mathbb{E} |V_N(s)|^2 \mathbb{E}\left(\sqrt{N}|Z_N(s)|\right)^2.\]
Both of the above terms are bounded in $N$ uniformly on $[0, t]$
by Lemma \ref{lem:moment-Z} and \ref{lem:moment-V}.
\end{proof}

Combining this result with Theorem \ref{thm:sens-rate}, the following corollary is immediate. 
\begin{corollary}\label{cor:rate-CGT}
For any given $t > 0$, 
the estimator variance of CGT method is asymptotically $\mathcal{O}(N^{2\alpha})$ uniformly on $[0,t]$.
\end{corollary}

%

\begin{theorem}\label{thm:rate-FD}
In addition to our running assumptions, we assume that $f$ in \eqref{eq-fN-limit} is continuously differentiable.
Then for each $t > 0$ and $h > 0$,
\[\sup_{s\leq t} \text{Var} \left(f^N(X^N(s, c+h)) - f^N(X^N(s, c))\right) = \mathcal{O}(N^{2\alpha - 1}).\]
That is, the estimator variance of FD method is asymptotically $\mathcal{O}(N^{2\alpha - 1})$.
\end{theorem}

\begin{proof}
Note that 
\begin{equation*}
\begin{split}
&\text{Var} \left(f^N(X^N(s, c+h)) - f^N(X^N(s, c))\right)\\ 
 \leq  &2\text{Var} \left(f^N(X^N(s, c+h))\right) + 2\text{Var} \left(f^N(X^N(s, c))\right).
\end{split}
\end{equation*}
Hence it is sufficient to show that $\text{Var} \left(f^N(X^N(t, c))\right) = \mathcal{O}(N^{2\alpha - 1})$ for any $c$. 
We write 
\[\frac{1}{N^{2\alpha - 1}}\text{Var} \left(f^N(X^N(s, c))\right) = N\mathbb{E}\left(  \left|\frac{f^N(X^N(s, c))}{N^{\alpha}} - \mathbb{E}\left(\frac{f^N(X^N(s, c))}{N^{\alpha}}\right)\right|^2 \right).\]
One can estimate the right hand side by using the same argument as is in Theorem \ref{thm:rate-CGT} to obtain an upper bound $8 L_K^2 + 8 C_K^2 \mathbb{E} \left(|V_N(s)|\right)^2$, which is bounded in $N$ uniformly on $[0, t]$ by Lemma \ref{lem:moment-V}.
\end{proof}

\begin{remark}\label{rem-RSD}
Based on Theorem \ref{thm:sens-rate}, Corollary \ref{cor:rate-GT}, Corollary \ref{cor:rate-CGT} and Theorem \ref{thm:rate-FD}, we may expect the RSDs of the GT, CGT and FD methods to scale as $\mathcal{O}(N^{1/2})$, $\mathcal{O}(1)$ and $\mathcal{O}(N^{-1/2})$, respectively. Since in Theorem \ref{thm:sens-rate},
 we do not have an exact limit for the sensitivity itself, this conclusion is not rigorously proven.  
As mentioned in Remark \ref{rem-sens-limit}, under additional assumptions, this conclusion will be true.
Our numerical results in the next section also support this statement. Moreover, we note that the $\mathcal{O}(N^{2 \alpha + 1})$ estimates in 
Theorem \ref{thm:GIR} and Corollary \ref{cor:rate-GT} are sharp. 
\end{remark}
%

\section{Numerical examples}\label{sec:numerical}
We illustrate the dependence of RSD of various sensitivity estimators (with respect to the deterministic parameter) on the system size $N$ via numerical examples.  
When comparing the GT or CGT methods with FD or RPD methods, we must bear in mind that while GT and CGT do not have method parameters, the FD method has a perturbation parameter $h$ and the RPD method has a window size parameter $w$, making the comparison not straightforward. Moreover, the FD and the RPD methods are biased.
A proper practical comparison involves choosing parameters $h$ and $w$ to
obtain an acceptable bias. We do not pursue such a detailed comparison here as
we are focused solely on the dependence on system size $N$. In the case of FD
or RPD methods, we fix $h$ or $w$ respectively, and vary $N$.   
We also use the CRN FD method instead of the IRN FD, as that is the more
commonly used approach. Moreover, since our variance estimates for FD methods 
were derived based on an upper bound which is twice that of the IRN FD method,
it is important to compare the performance of CRN FD to see if the order
estimate $\mathcal{O}(N^{-1/2})$ for the RSD is sharp.  

We note that in the very large system size limit, the stochastic system behaves nearly deterministically and hence none of these stochastic sensitivity methods are
needed; traditional ODE sensitivity methods would do. However, when the system size $N$ is modestly large, say $N = 100$, the system may not be approximated by the ODE and our asymptotic analysis may be relevant in this regime. Our numerical results below show this.  
\subsection{Numerical example 1}  
The reversible isomerization model consists of two species $S_1$ and $S_2$ and involves the following two reactions:

\begin{equation}
S_1 \xrightarrow{c_1} S_2,    \qquad S_2 \xrightarrow{c_2} S_1.
\end{equation}
In the model with system size $N$, the intensity functions for processes $R_1^N$ and $R_2^N$ are 
\[a_1^N(X^N(t), c) =  c_1 X_1^N(t),\]
\[a_2^N(X^N(t), c) =  c_2 X_2^N(t),\]
respectively. The stoichiometric vectors are $\nu_1 = [-1, 1]^T$ and $\nu_2 =[1, -1]^T$.

In this example, the expectation of the population of species at a fixed time $t$ can be computed analytically:
\begin{equation}\label{equ:rever-isom-1}
E[X_1^N(t)]=X_1^N(0) + \frac{1-e^{-(c_1+c_2)t}}{c_1+c_2}(c_2X_2^N(0)-c_1X_1^N(0)),
\end{equation}
\begin{equation}\label{equ:rever-isom-2}
E[X_2^N(t)]=X_2^N(0) - \frac{1-e^{-(c_1+c_2)t}}{c_1+c_2}(c_2X_2^N(0)-c_1X_1^N(0)),
\end{equation}
where $X_1^N(0)$ and $X_2^N(0)$ are assumed to be deterministic. 
One can compute the exact sensitivities by differentiating \eqref{equ:rever-isom-1} and \eqref{equ:rever-isom-2} with respect to parameters.
In the numerical tests considered here, we choose parameters $c_1 = 0.3$ and
$c_2 = 0.2$ and the initial population $X_1^N(0) = N$ and $X_2^N(0) = N$,
where $N$ is the system size parameter. We set the terminal time $T = 10$ and
compute the sensitivity for $N = 1,2,5,10,50, 100,200,300, 400, 500, 600, 700,
800, 900$ and $1000$. We use four different methods here, namely GT, CGT, CRN
FD and RPD. 
We note that by CRN FD, we mean the common random number and one-sided finite
difference method in conjunction with Gillespie's SSA \cite{CRP}. 
The perturbation parameter for the CRN FD method is $h = 0.01$ for parameter $c_1$ and the window size parameter $w = 1.0$ for RPD method for terminal time $T = 10$. The number of trajectories for simulation is $N_s = 10^6$ for each system size $N$. We consider sensitivities with respect to $c_1$ of the expected values of four different output functions.
   
The first output function we consider here is 
$f^N(x) = x_1$ for all $N$, that is, we compute the sensitivity of $\mathbb{E}(X_1^N(T))$ with respect to parameter $c_1$. Obviously, conditions in Assumption \ref{assume5} are satisfied with $\alpha = 1$ and $f(x) = x_1$. We examine the growth of sensitivity of $ \mathbb{E}(X_1^N(T))$ with respect to $c_1$ in terms of $N$ using  $10^6$ independent trajectories. The computed sensitivity and the error in the sensitivity estimate
are shown in Figure \ref{fig:subfig:mean_reverse_isom}, and Figure \ref{fig:subfig:var_reverse_isom} shows the loglog plot of RSD of all four methods.

\begin{figure}[!ht]
\centering
\subfigure[Sensitivity]
{ \label{fig:subfig:mean_reverse_isom}
\includegraphics[width=2.4in]{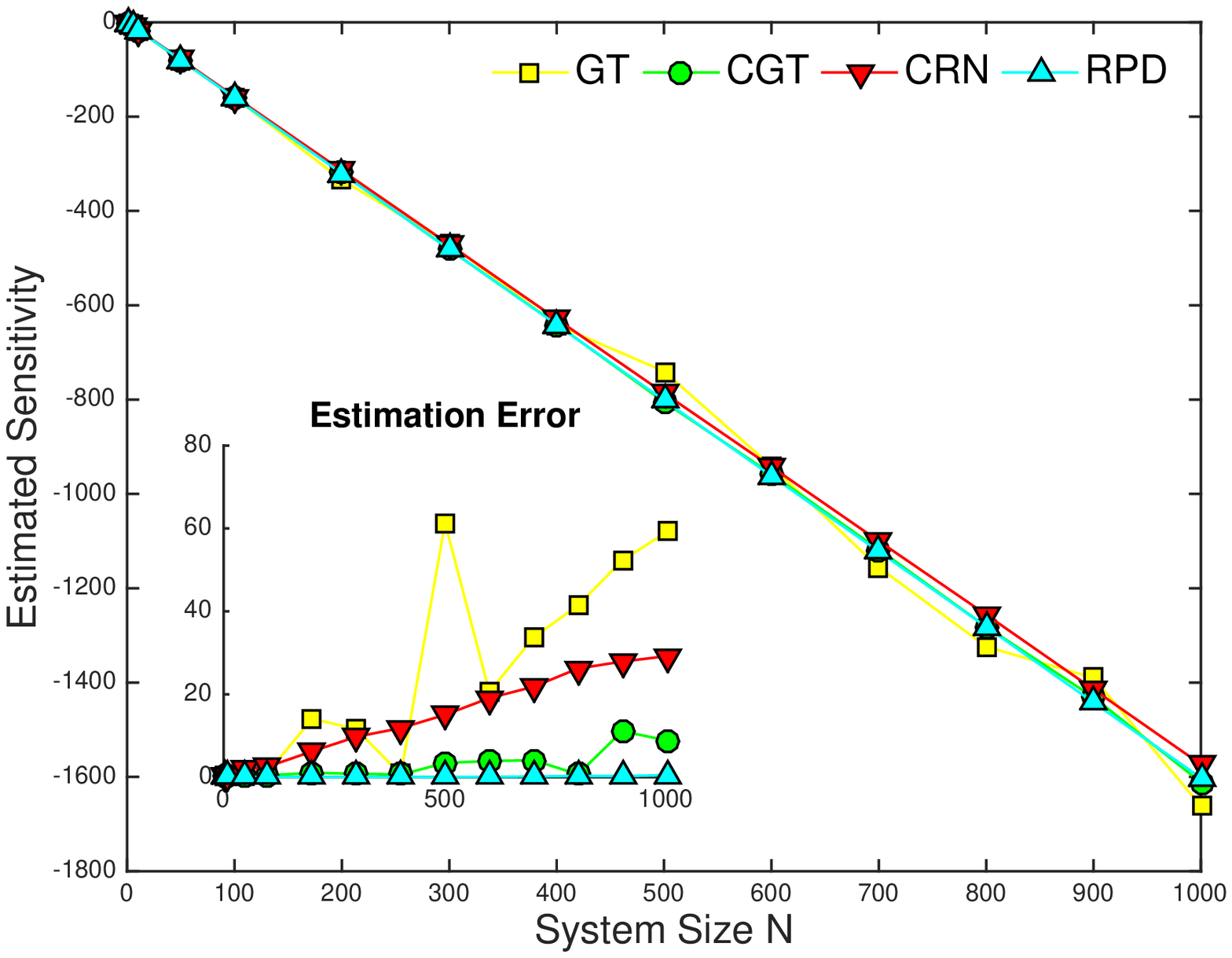}}
\subfigure[RSD]
{ \label{fig:subfig:var_reverse_isom}
\includegraphics[width=2.4in]{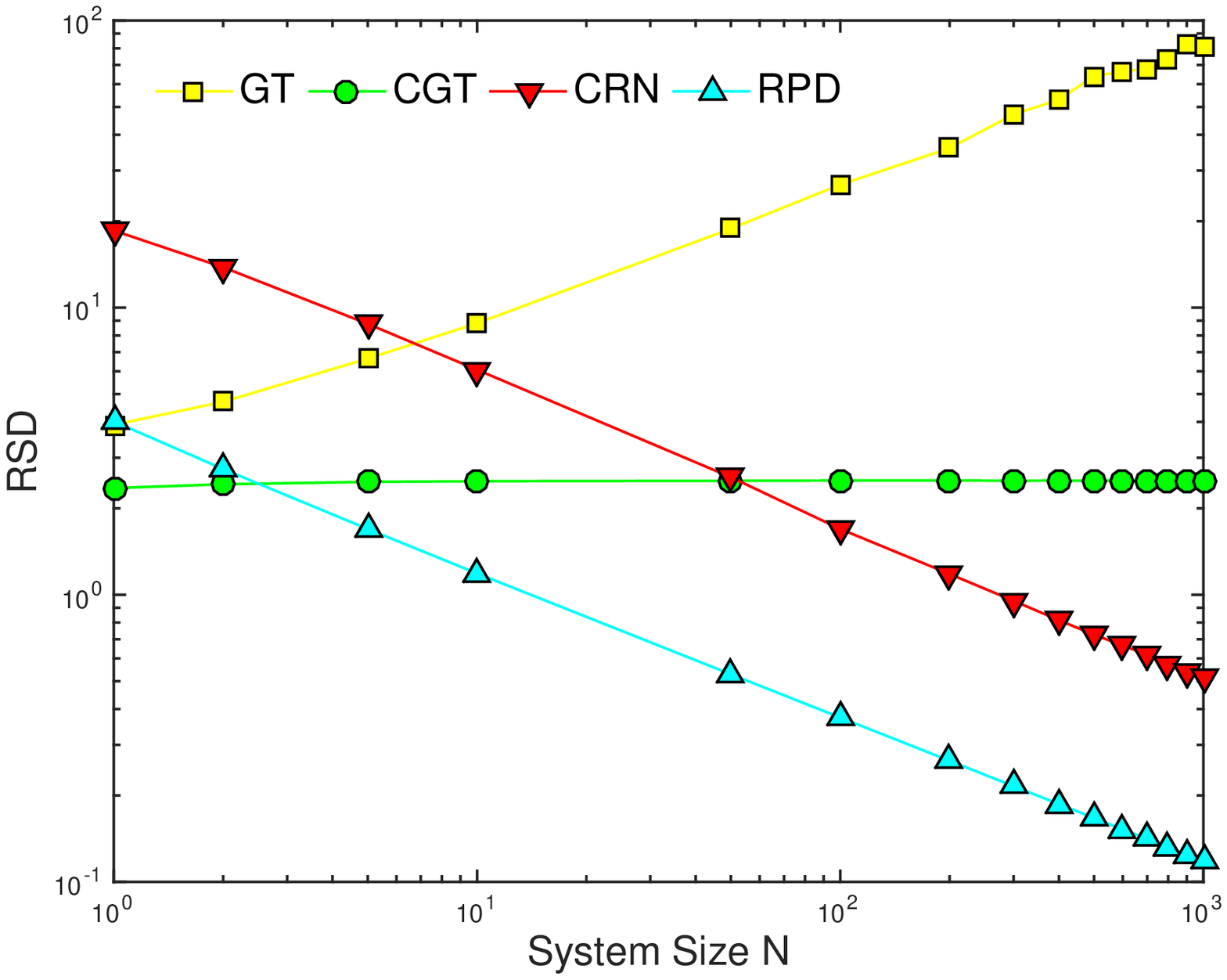}}
\caption{Estimated sensitivity (left) and error in the sensitivity estimate (inset) of $\mathbb{E} (X_1^N(T))$ with respect to $c_1$, and RSD (right) at terminal time $T = 10$ for reversible isomerization model.}
\label{fig:reverse_isom}
\end{figure}


\begin{figure}[!ht]
\centering
\subfigure[Sensitivity]
{ \label{fig:subfig:mean_reverse_isom_poly}
\includegraphics[width=2.4in]{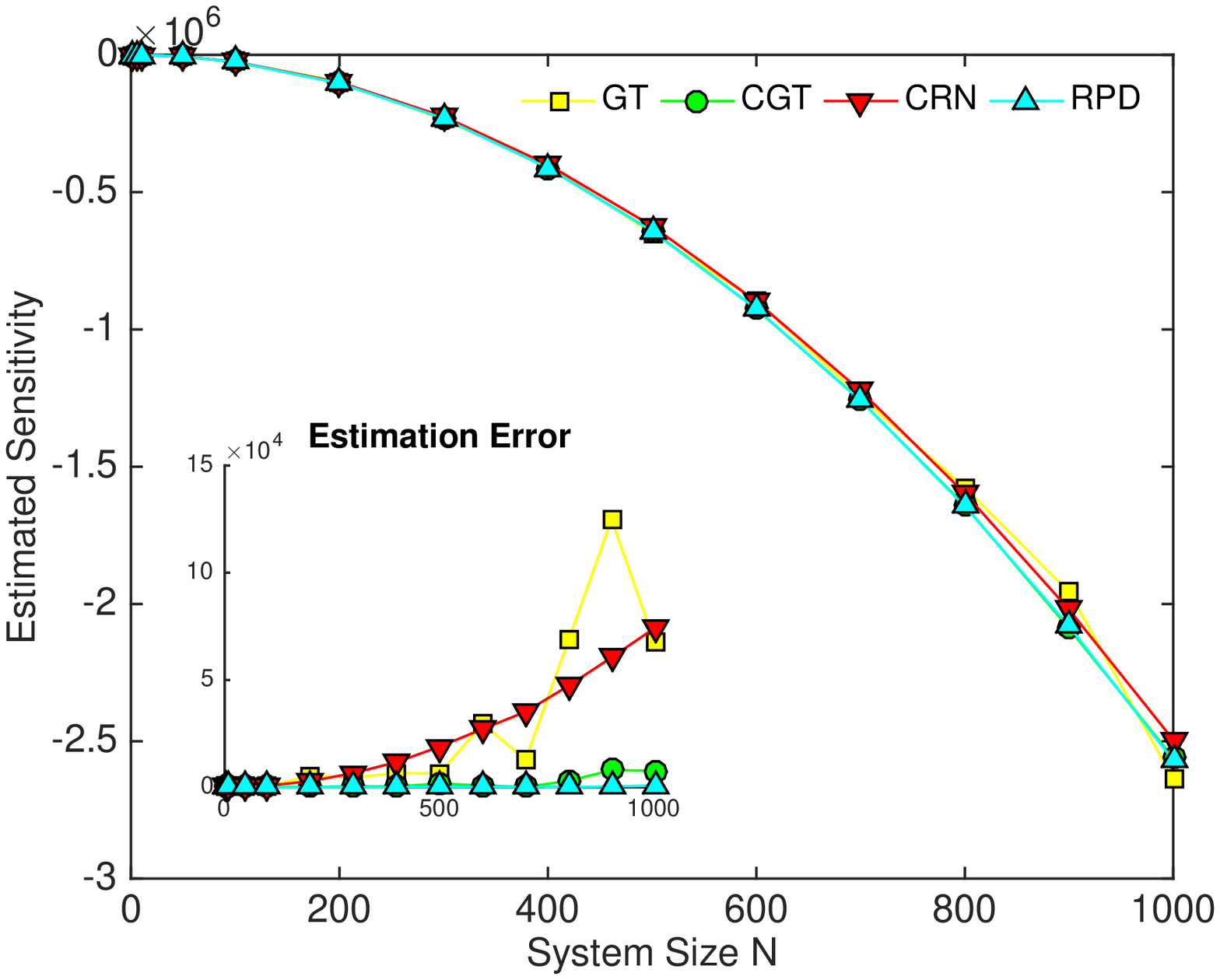}}
\subfigure[RSD]
{ \label{fig:subfig:var_reverse_isom_poly}
\includegraphics[width=2.4in]{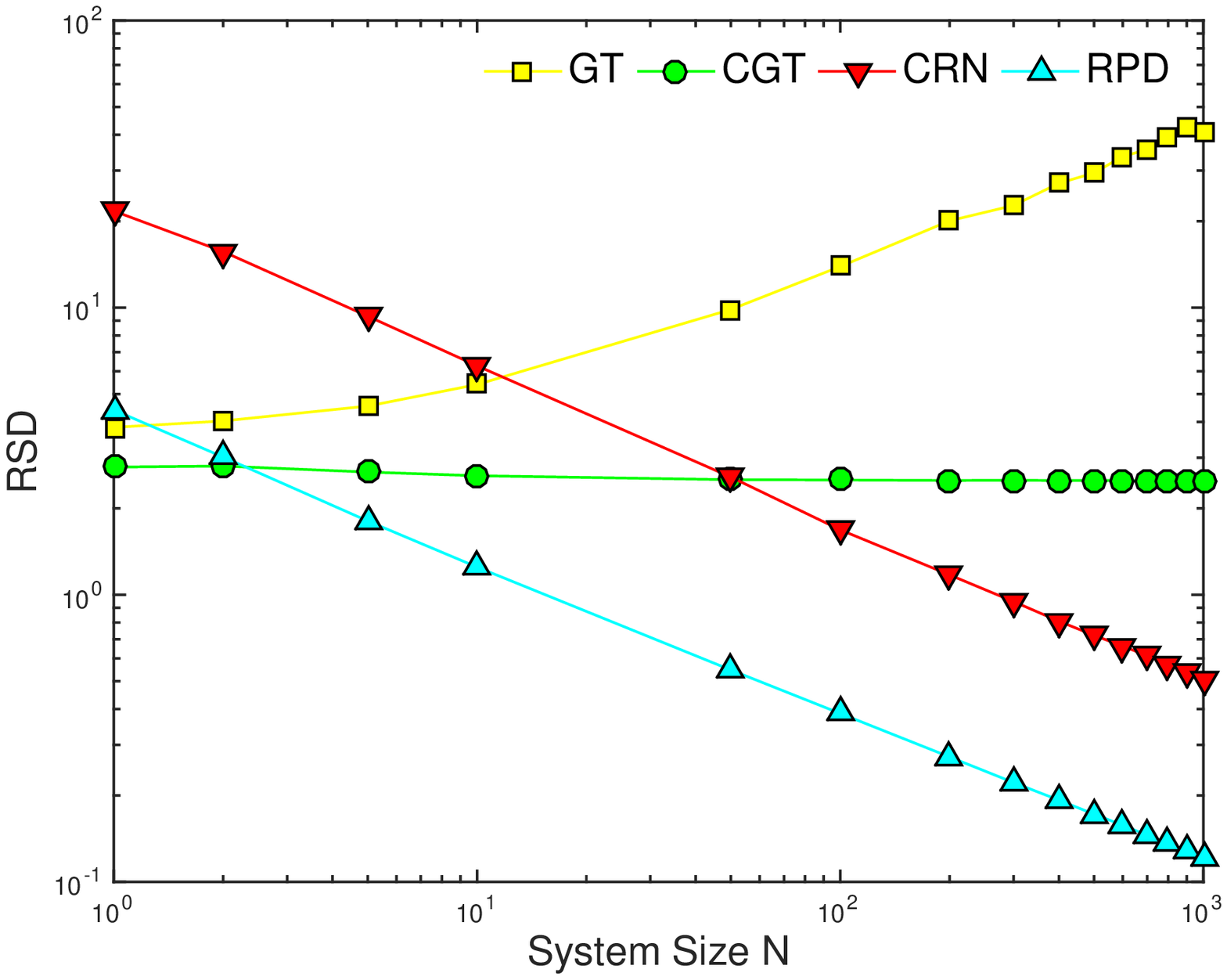}}
\caption{Estimated sensitivity (left) and error in the sensitivity estimate (inset) of $\mathbb{E} (X_1^N(T))^2$ with respect to $c_1$ and RSD (right) at terminal time $T = 10$ for reversible isomerization model.}
\label{fig:reverse_isom_poly}
\end{figure}

\begin{figure}[!ht]
\centering
\subfigure[Sensitivity]
{ \label{fig:subfig:mean_reverse_isom_sin}
\includegraphics[width=2.4in]{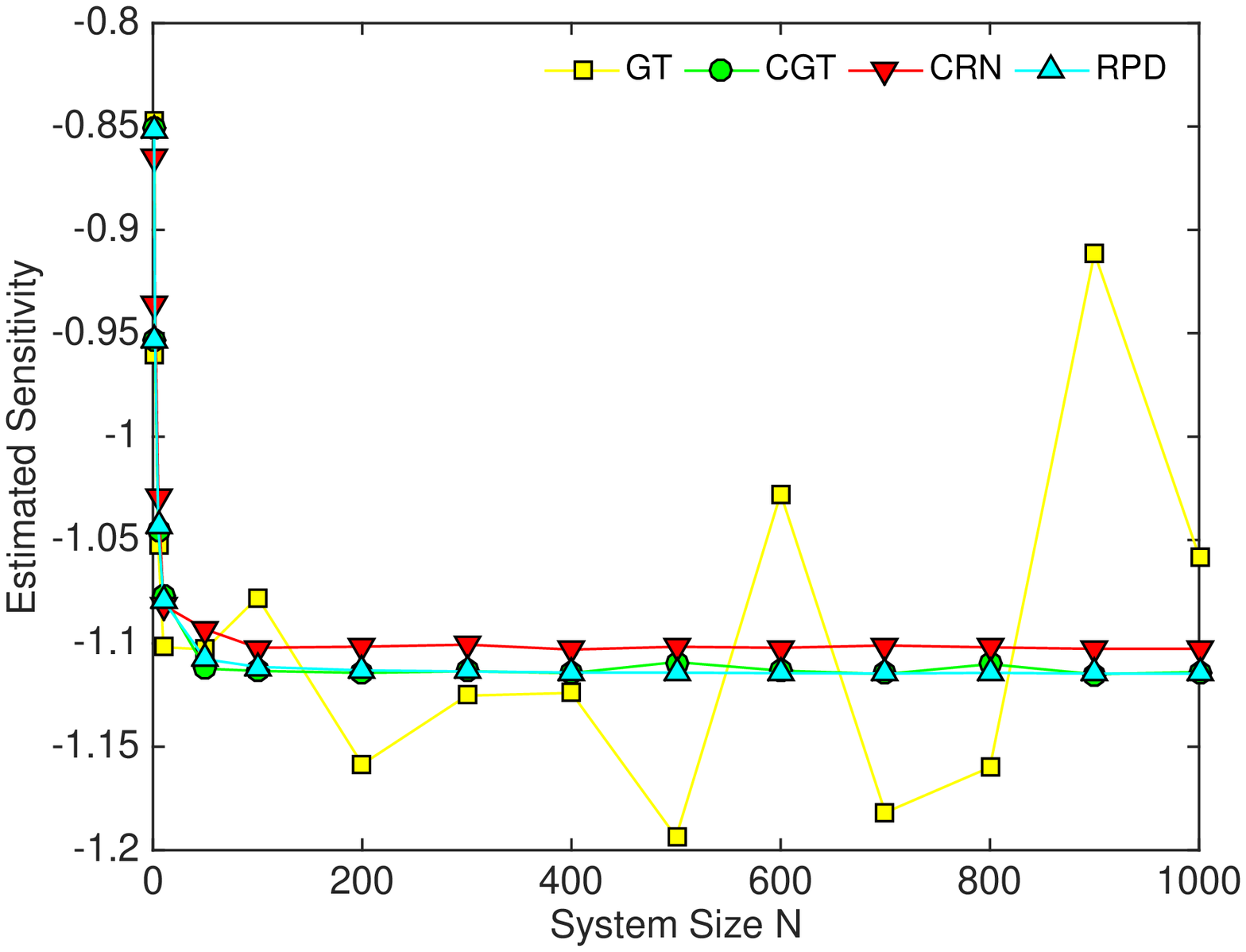}}
\subfigure[RSD]
{ \label{fig:subfig:var_reverse_isom_sin}
\includegraphics[width=2.4in]{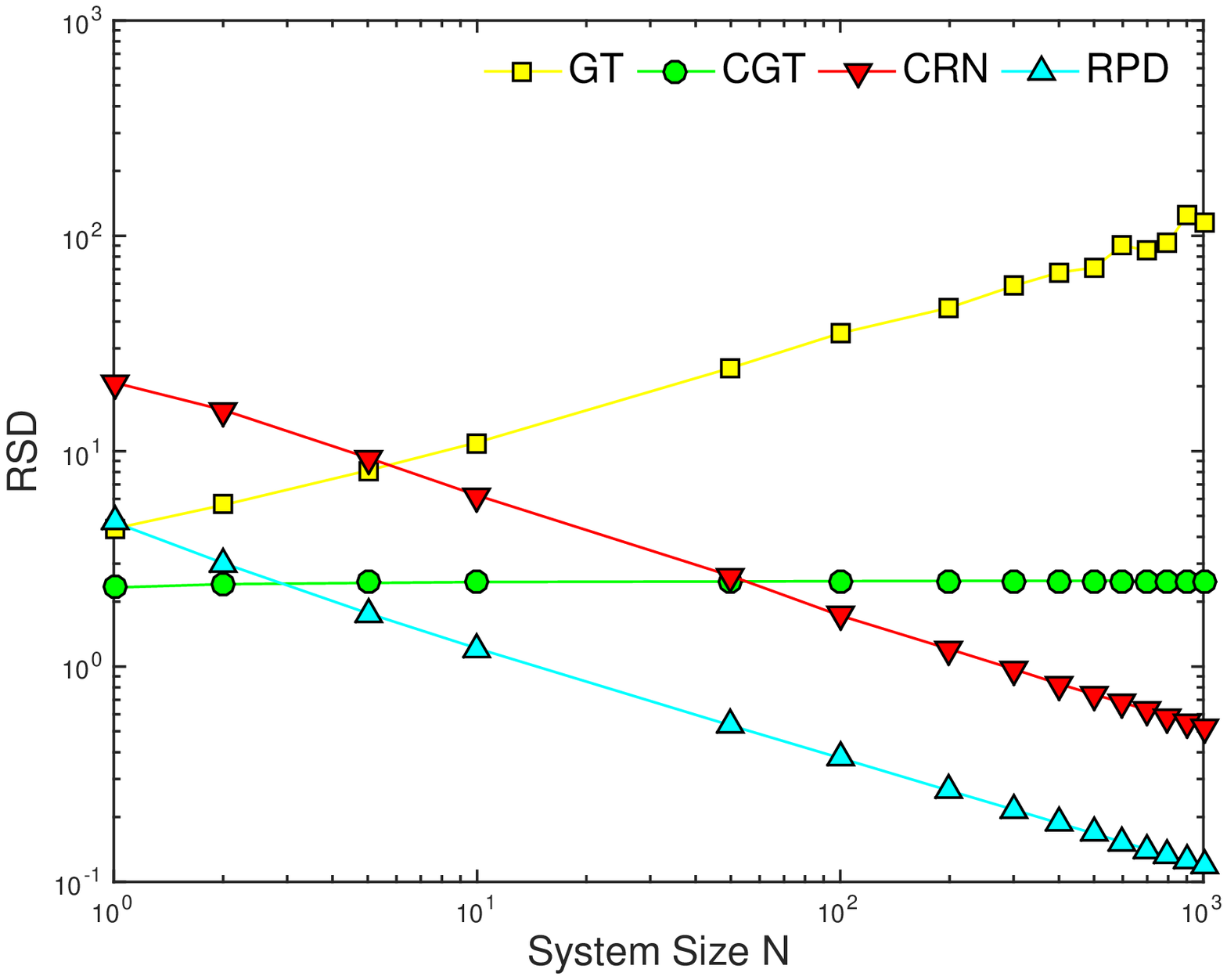}}
\caption{Estimated sensitivity of $\mathbb{E} (\sin({X_1^N(T)}/{N}))$ with respect to $c_1$ (left) and RSD (right) at terminal time $T = 10$ for reversible isomerization model.}
\label{fig:reverse_isom_sin}
\end{figure}

The second output function we use for testing is $f^N(x) = x_1^2$ for all $N$. By \eqref{eq-fN-limit}, $f(x) = x_1^2$ and $\alpha = 2$ in Assumption \ref{assume5}. Similar to the case of output function $f^N(x) = x_1$, the exact sensitivity in this case can be calculated and hence we show the error in the sensitivity estimate as an inset plot. See Figure \ref{fig:reverse_isom_poly} for sensitivity and RSD. 
The third output function we consider is $f^N(x) = \sin(x_1/N)$ and so $f(x) = \sin x_1$. It can be seen that for this case, $\alpha = 0$ in Assumption \ref{assume5}. 
Plot for the numerical result is shown in Figure \ref{fig:reverse_isom_sin}.

The last output function we consider here is the indicator function $f^N(x) = 1_{\{x_1 \leq x_2\}}(x)$, which does not satisfy the conditions in our theorems since $f = 1_{\{x_1 \leq x_2\}}$ is not continuously differentiable. However, numerical tests still show similar behavior as indicated by our theorems. Note that the sensitivity approaches to zero as $N$ increases to $\infty$ and hence RSD is not well defined for large $N$. Instead, we plot the estimator variance against $N$ in Figure \ref{fig:subfig:var_reverse_isom_ind}. 
\begin{figure}[!ht]
\centering
\subfigure[Sensitivity]
{ \label{fig:subfig:mean_reverse_isom_ind}
\includegraphics[width=2.4in]{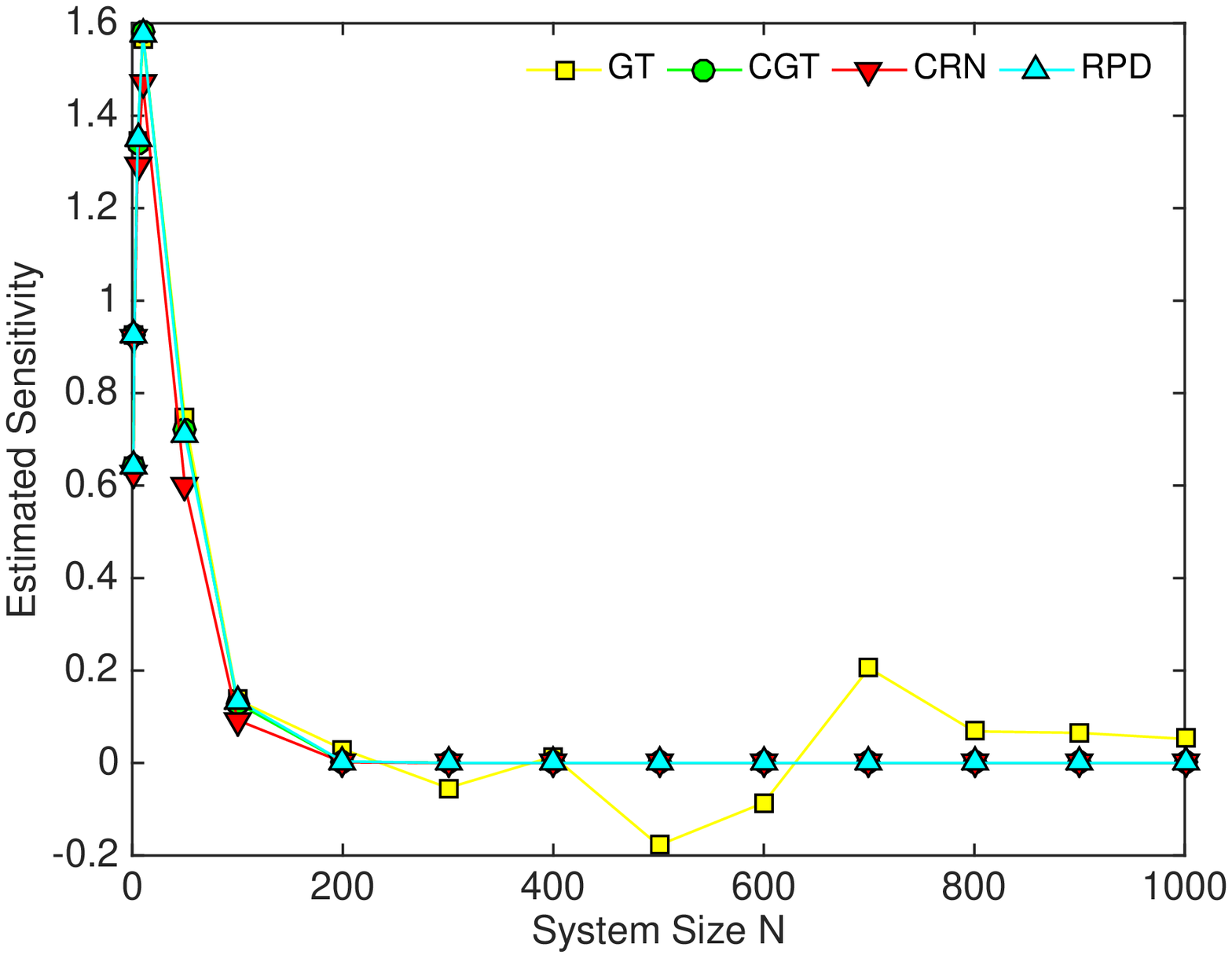}}
\subfigure[Variance]
{ \label{fig:subfig:var_reverse_isom_ind}
\includegraphics[width=2.4in]{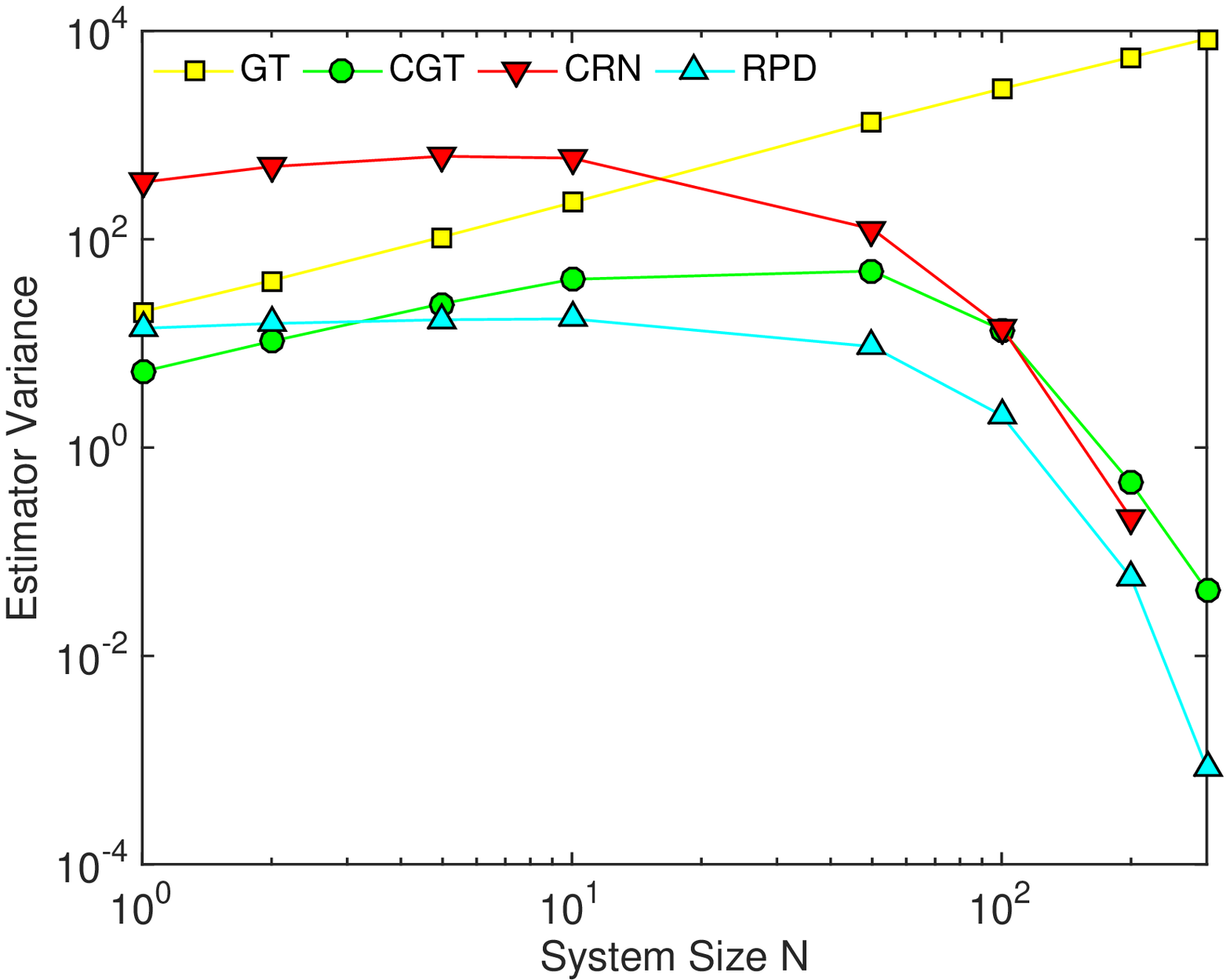}}
\caption{Estimated sensitivity of $\mathbb{P} (X_1^N(T) \leq X_2^N(T))$ with respect to $c_1$ (left) and variance (right) at terminal time $T = 10$ for reversible isomerization model.}
\label{fig:reverse_isom_ind}
\end{figure}

Additionally, Table \ref{tab:reverse-isom-slope} summarizes the rate of growth (as a power of $N$) of the numerically estimated RSD for the different estimators considered above. The results are in agreement with the theory. 




\begin{table}[!ht]
\caption{Observed slopes (via regression for large $N$) for the loglog plots of RSD 
for reversible isomerization model, that is,  $R_1$, $R_2$ and $R_3$ are the observed asymptotic order of the estimator RSD (as a power of $N$) for $\mathbb{E}(X_1^N(T))$, $\mathbb{E}(X_1^N(T))^2$ and $\mathbb{E}(\sin (X_1^N(T)/ N))$, respectively.}
\centering
\begin{tabular}{| c | c | c | c |}
\hline     &  $R_1$  & $R_2$   &  $R_3$ \\
\hline
\hline
GT  &  0.4992                  &  0.4895                      &  0.5724              \\
\hline
CGT &  -0.0004                  &  -0.0008                      &  0.0009             \\   
\hline
CRN FD &  -0.5156                  &  -0.5160                     & -0.5162             \\   
\hline
RPD &  -0.5005                  & -0.5000                      &  -0.5000              \\
\hline
\end{tabular}
\label{tab:reverse-isom-slope}
\end{table}


\subsection{Numerical example 2}
As a second numerical example, let us consider the decaying-dimerizing model \cite{Gillespie-tau-leaping}
\begin{equation}\label{equ: decay-dim}
S_1 \xrightarrow{c_1} \emptyset,  \qquad 2S_1 \xrightarrow{c_2} S_2,
\qquad S_2 \xrightarrow{c_3} 2S_1,
\qquad S_2 \xrightarrow{c_4} S_3.
\end{equation}

The stoichiometric vectors are $\nu_1 = [-1, 0, 0]^T$, $\nu_2 = [-2, 1, 0]^T$, $\nu_3 = [2, -1, 0 ]^T$ and $\nu_4 = [0, -1, 1]^T$.  
We set the initial population to be $X_1^N(0) = 10N, X_2^N(0) = 0, X_3^N(0) = 0$.
Using the stochastic mass action form \eqref{equ:mass action}, the intensity for processes $R_1^N$, $R_2^N$, $R_3^N$ and $R_4^N$ are 
\[a_1^N(X^N(t), c) = c_1X_1^N(t),\] 
\[a_2^N(X^N(t), c) = \frac{c_2}{2 N} X_1^N(t)(X_1^N(t) - 1), \] 
\[a_3^N(X^N(t), c) = c_3 X_2^N(t), \]  
\[a_4^N(X^N(t), c) = c_4 X_2^N(t).\]
We set the parameters as follows, $c_1 = 1.0$, $c_2 =0.002$, $c_3 = 0.5$ and $c_4 = 0.04$. Note that the intensity for the second reaction is not linear, hence an analytical formula for the sensitivity is not attainable.  
We test the sensitivity and RSD for $\mathbb{E}[f^N(X_1^N)]$ with respect to $c_1$.
For the CRN FD method, we use one-sided finite difference scheme and perturb
the parameter $c_1$ by $h = 0.01$. Note that RPD is not applicable for this
example since the firing of the first reaction will prevent the second
reaction to happen when the population of $S_1$ is $1$ (see
\cite{RPD}). Therefore, we only examine the RSDs of GT, CGT and CRN FD here. For each system size $N$, the number of trajectories we use for simulation is $\text{N}_s = 10^6$. Plots of the sensitivity and RSD are shown in Figure \ref{fig:decay_dim}, \ref{fig:decay_dim_poly} and \ref{fig:decay_dim_sin} for $\mathbb{E} (X_1^N(T))$, $\mathbb{E} (X_1^N(T))^2$ and $\mathbb{E} (\sin (X_1^N(T)/N))$, respectively. The rate of growth (as a power of $N$) of the numerically estimated RSD are summarized in Table \ref{tab:decay-dim-slope}.
\begin{figure}[!ht]
\centering
\subfigure[Sensitivity]
{ \label{fig:subfig:mean_decay_dim}
\includegraphics[width=2.4in]{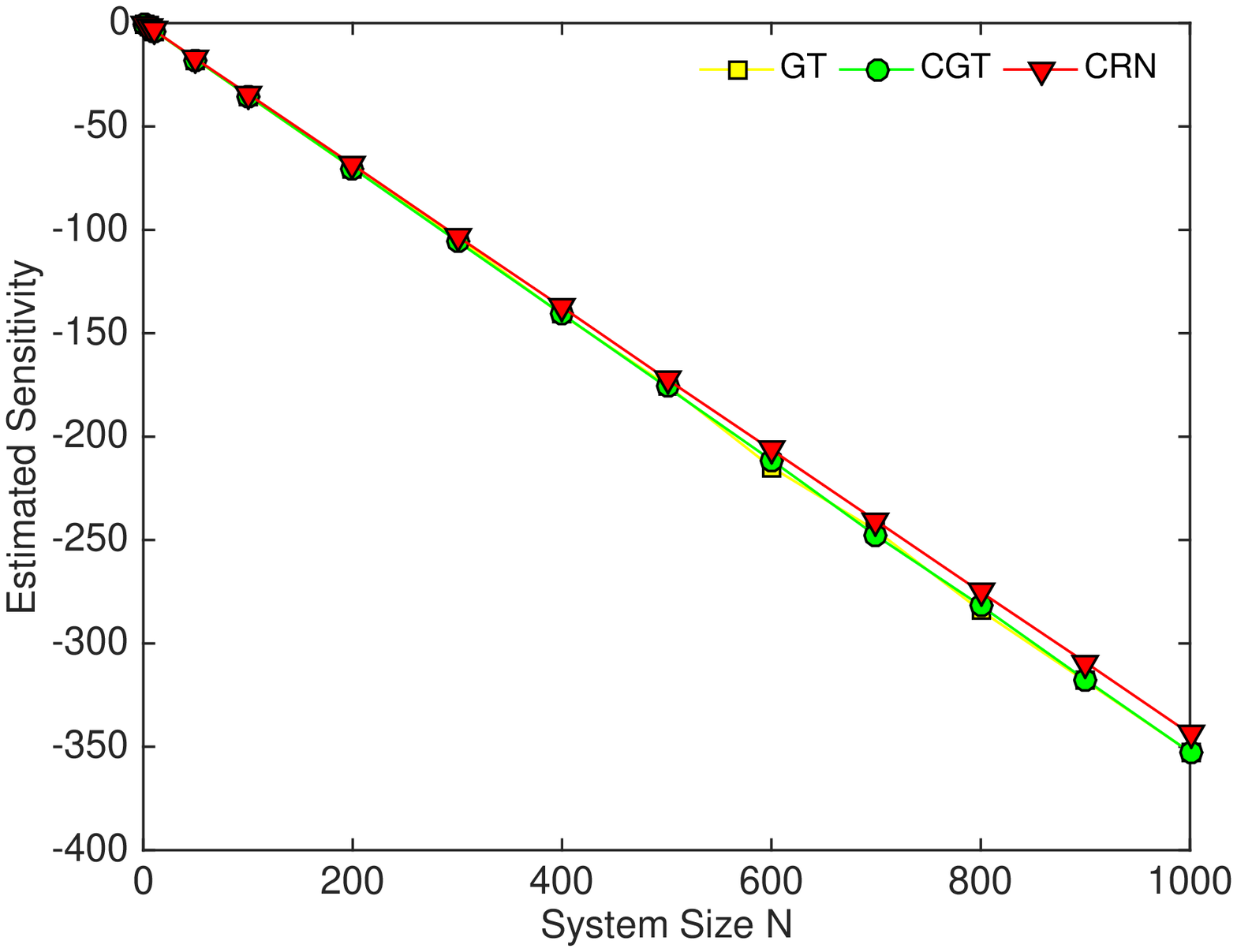}}
\subfigure[RSD]
{ \label{fig:subfig:var_decay_dim}
\includegraphics[width=2.4in]{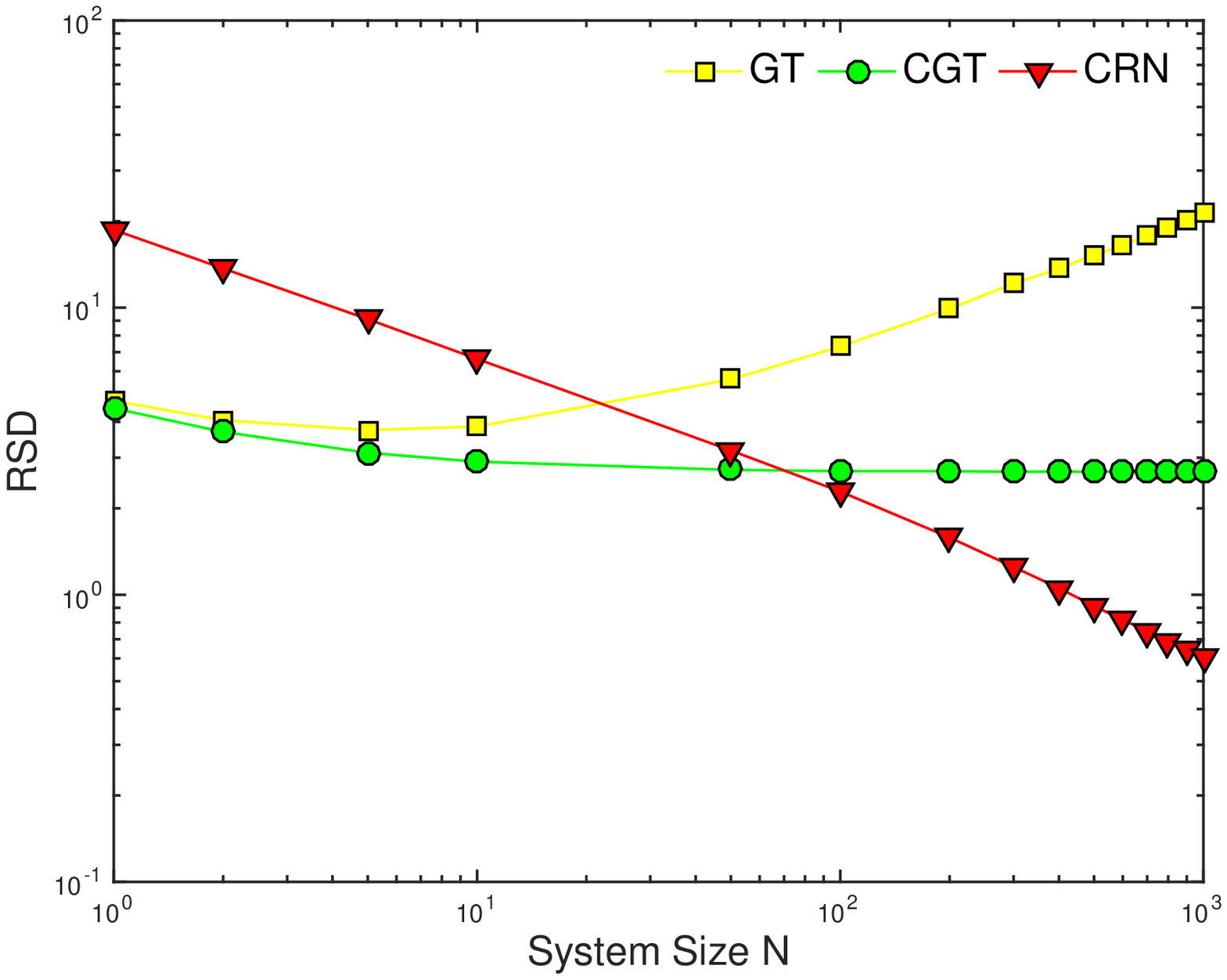}}
\caption{Estimated sensitivity of $\mathbb{E} [X_1^N(T)]$ with respect to $c_1$ and RSD at terminal time $T = 5$ for decaying-dimerizing model.}
\label{fig:decay_dim}
\end{figure}

\begin{figure}[!ht]
\centering
\subfigure[Sensitivity]
{ \label{fig:subfig:mean_decay_dim_poly}
\includegraphics[width=2.4in]{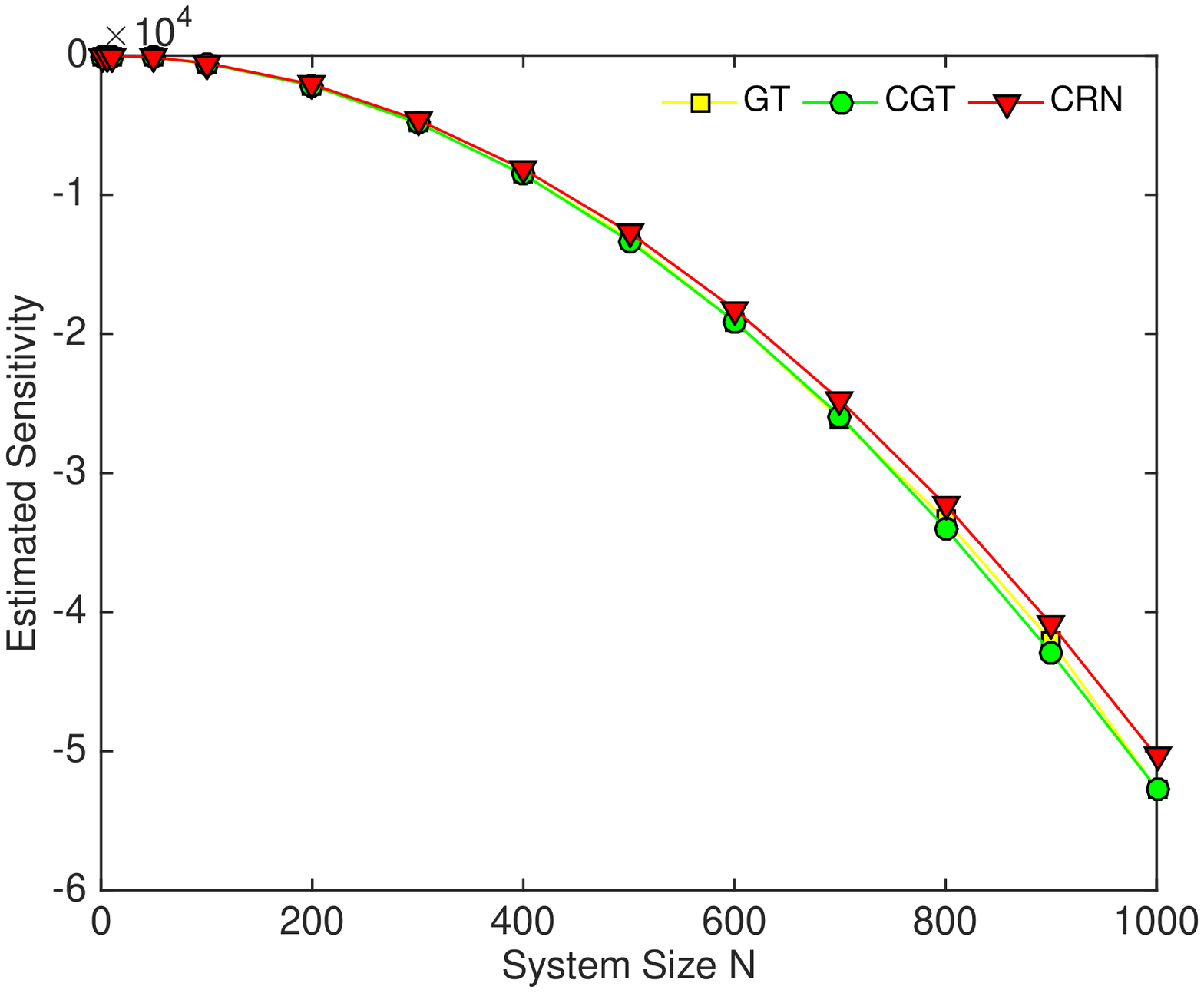}}
\subfigure[RSD]
{ \label{fig:subfig:var_decay_dim_poly}
\includegraphics[width=2.4in]{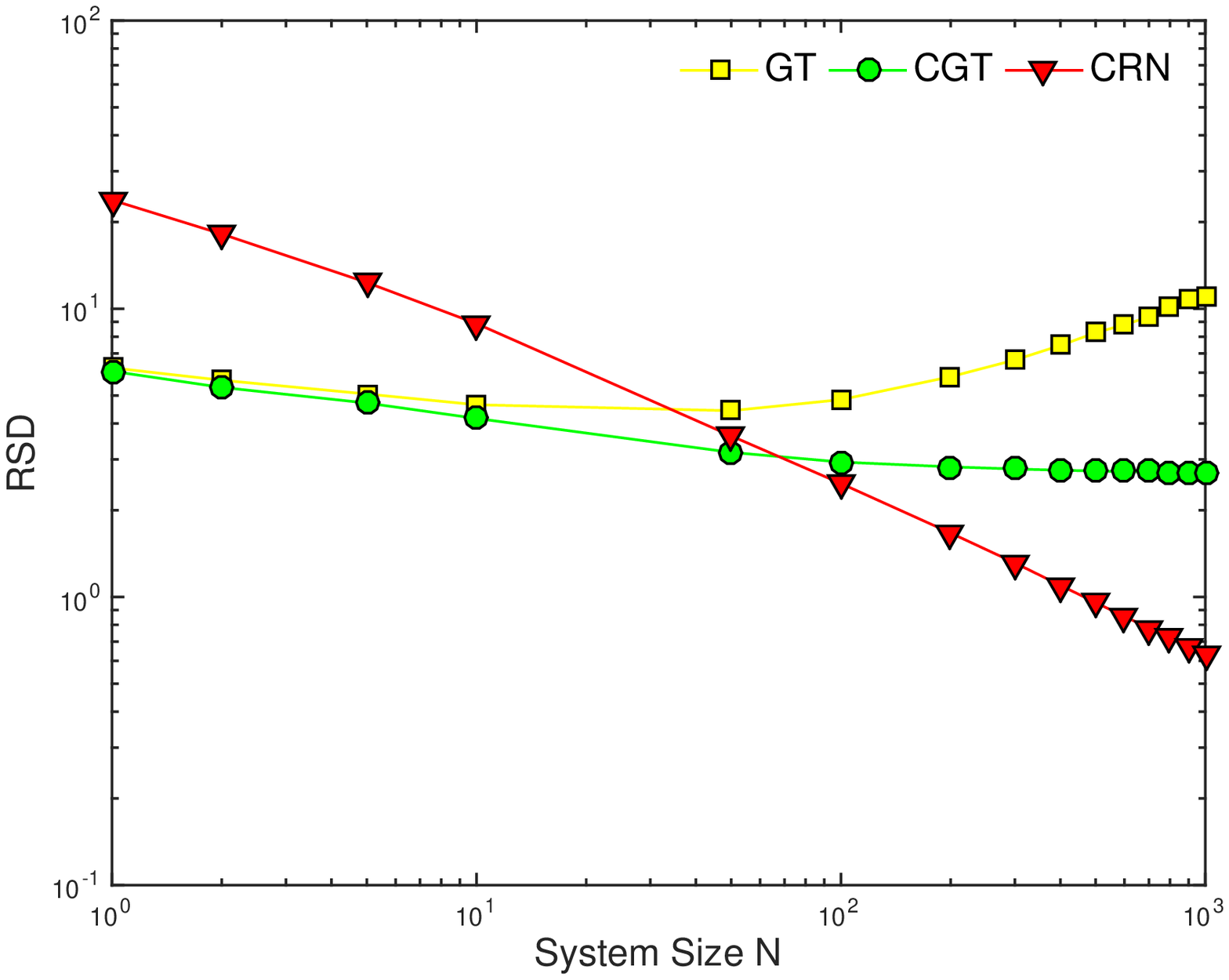}}
\caption{Estimated sensitivity of  $\mathbb{E} (X_1^N(T))^2$ with respect to $c_1$ and RSD at terminal time $T = 5$ for decaying-dimerizing model.}
\label{fig:decay_dim_poly}
\end{figure}

\begin{figure}[!ht]
\centering
\subfigure[Sensitivity]
{ \label{fig:subfig:mean_decay_dim_sin}
\includegraphics[width=2.4in]{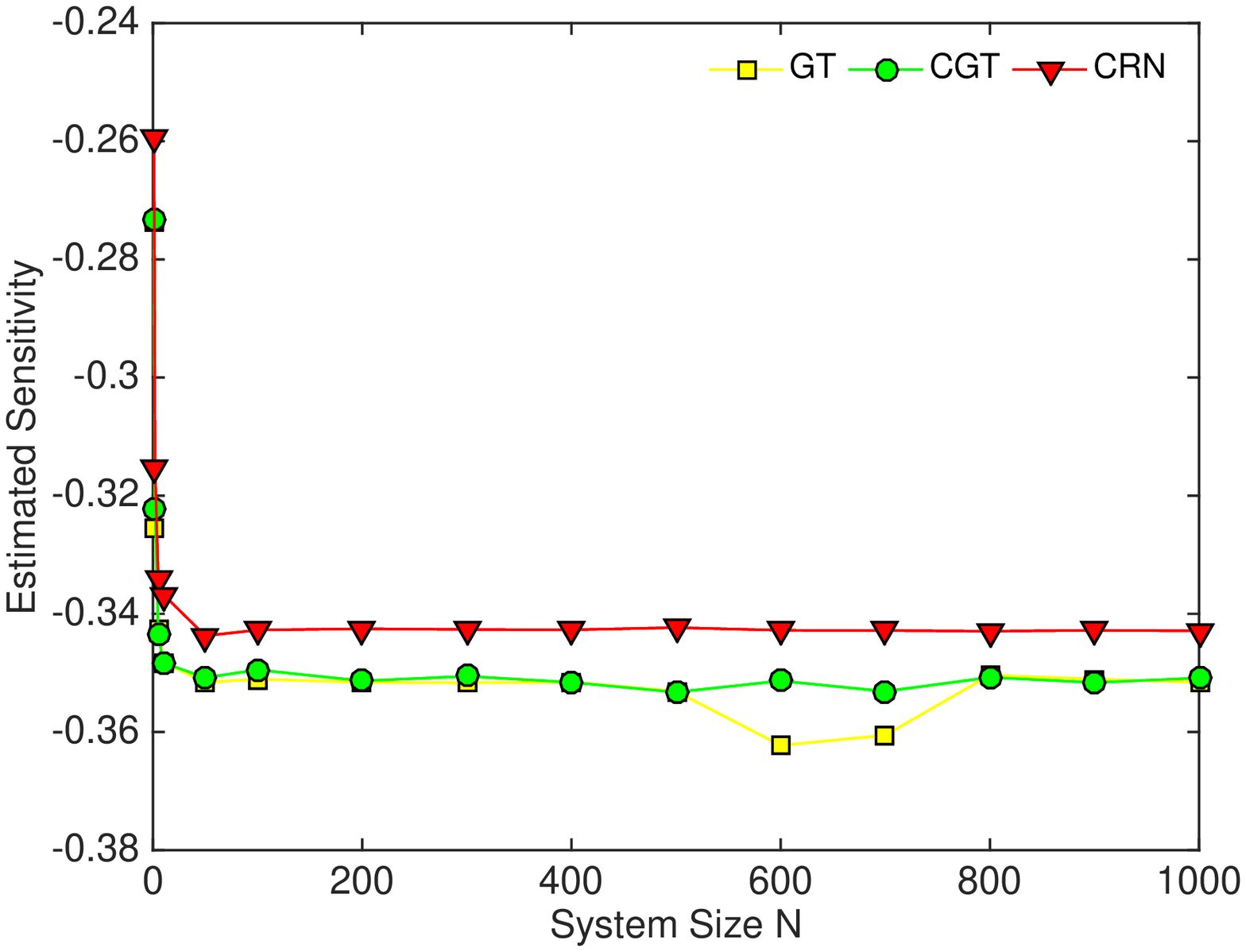}}
\subfigure[RSD]
{ \label{fig:subfig:var_decay_dim_sin}
\includegraphics[width=2.4in]{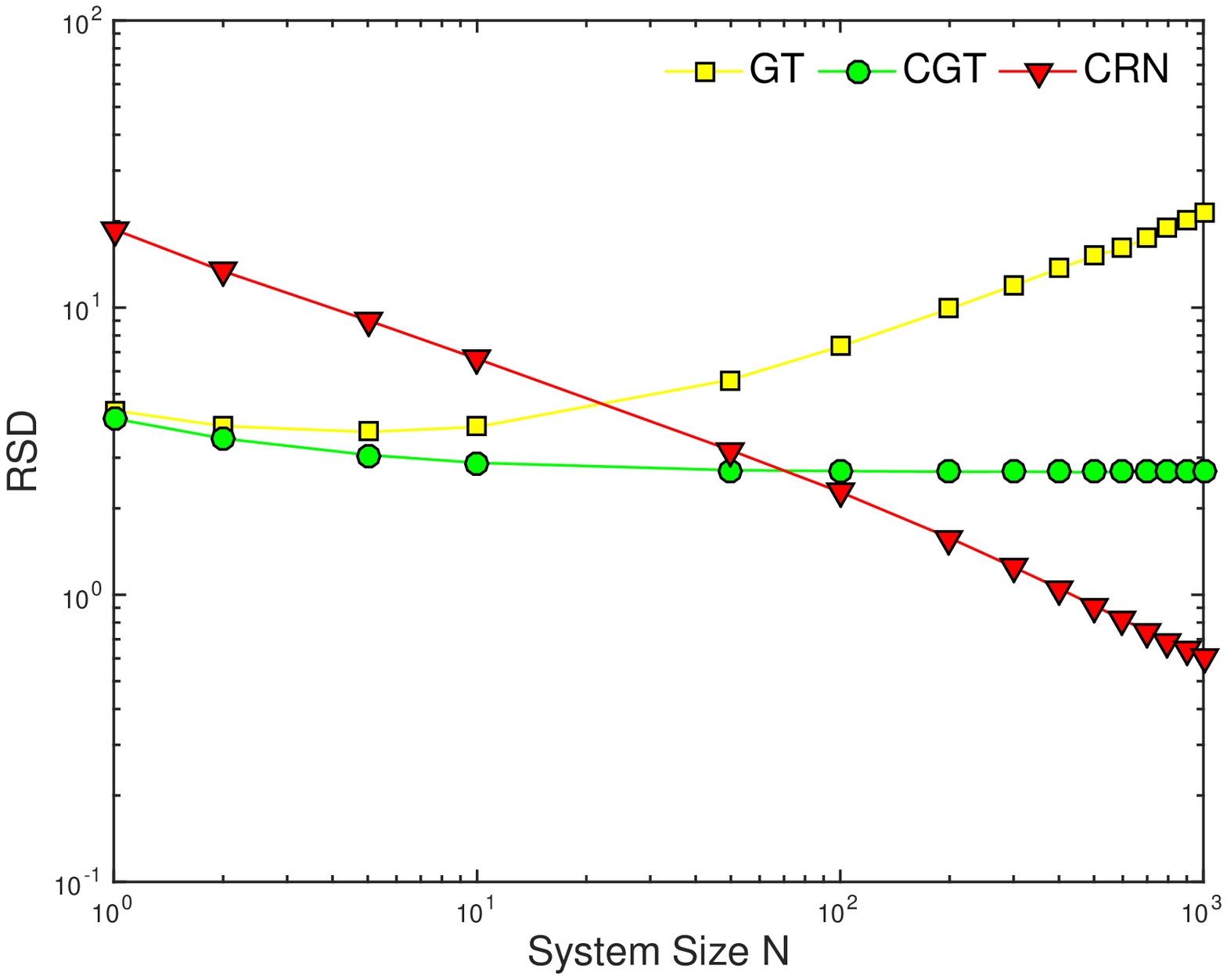}}
\caption{Estimated sensitivity of $\mathbb{E} (\sin (X_1^N(T)/N))$ with respect to $c_1$ and RSD at terminal time $T = 5$ for decaying-dimerizing model.}
\label{fig:decay_dim_sin}
\end{figure}

\begin{table}[!ht]
\caption{Observed slopes (via regression) for the loglog plots for RSD 
for decaying-dimerizing model, that is,  $R_1$, $R_2$ and $R_3$ are the observed asymptotic order of the estimator RSD (as a power of $N$) for $\mathbb{E}(X_1^N(T))$, $\mathbb{E}(X_1^N(T))^2$ and $\mathbb{E}(\sin (X_1^N(T)/ N))$, respectively.}
\centering
\begin{tabular}{| c | c | c | c |}
\hline
       &  $R_1$  &  $R_2$  &  $R_2$  \\
\hline
\hline
GT   &  0.4689                 &  0.4100                     &  0.4737      \\
\hline
CGT &  -0.0040                  &  -0.0257                     &  -0.0008        \\   
\hline
CRN FD &  -0.6022                  &  -0.6068                    & -0.6009       \\
\hline
\end{tabular}
\label{tab:decay-dim-slope}
\end{table}

\subsection{Numerical example 3}
In this numerical example, we revisit the reversible isomerization network to illustrate the asymptotic behavior of various estimators in terms of the terminal time $T$. Note that in this example, the deterministic parameters $c_j$ and the stochastic parameters $c_j'$ are the same.
For ease of notation, we suppress $N$ because we fix $N=10$ and only let $T$ change in this simulation.
The initial population is $X_1(0) = 10$ and $X_2(0)=10$. 
Parameters are taken to be $c_1 = 0.3$ and $c_2 = 0.2$ as before. 
In this case, the exact sensitivity can be obtained by taking derivative with respect to $c_1$ for \eqref{equ:rever-isom-1}. 
Figure \ref{fig:subfig:Sens_T_limit} shows the sensitivities estimated by GT,
CGT and CRN FD against the true sensitivity as a function of $T$.
The Figure \ref{fig:subfig:Var_T_limit} shows the estimator variances as a function of $T$. It can be seen that all three estimators show a variance that grows linearly in $T$ for the range of values of $T$ considered here.

In fact, this observation can be justified for the GT and CGT methods as follows.
Recall the definition of the centered processes $M_j(t) = R_j(t) - \int_0^t a_j(X(s))ds, j = 1, \cdots, m$.
Since $X_j(t)$ are bounded in this network, one can show that
\[\mathbb{E}M_j^2(t) = \mathbb{E}([M_j, M_j](t)) = \mathbb{E}R_j(t) = c_j\int_0^t \mathbb{E}X_j(s)ds = \mathcal{O}(t),\]
where the first equality holds since $M_j(t), j = 1, 2$ is a $L^2$-bounded martingale (see \cite{Protter}).
Therefore, we conclude that $\mathbb{E}Z^2(t) = \mathcal{O}(t)$ because in this case $Z(t) = c_1^{-1}M_1(t)$ and hence the variances of both GT and CGT are of $\mathcal{O}(t)$.

As for the variance of the FD estimator, the observed growth is approximately
linear in $t$ in the range of $10$ to $20$. However, from the upper bound 
used in the proof of Theorem \ref{thm:rate-FD}, it is easy to see that 
the estimator variance remains bounded as $t \to \infty$.

\begin{figure}[!ht]
\centering
\subfigure[Sensitivity]
{ \label{fig:subfig:Sens_T_limit}
\includegraphics[width=2.4in]{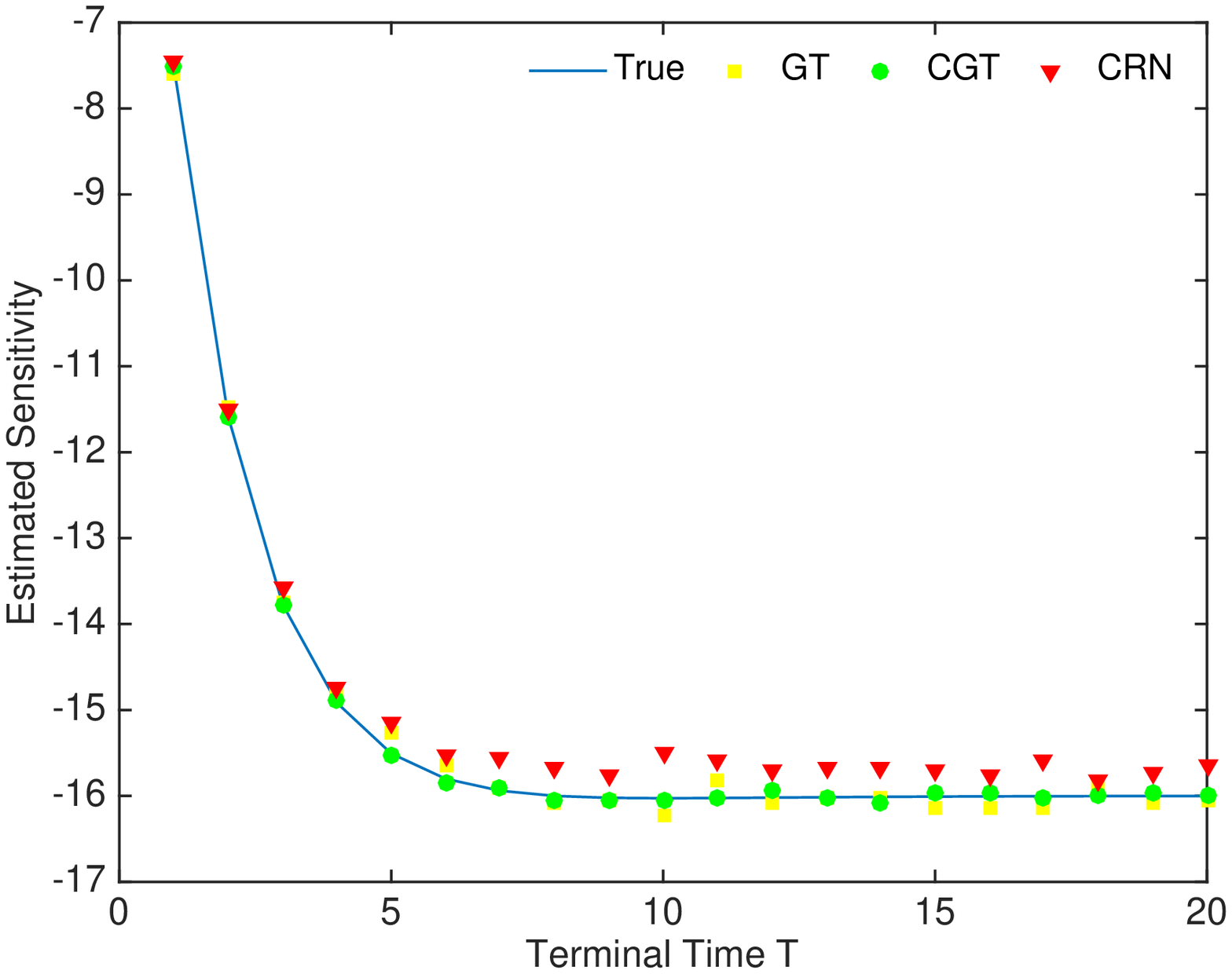}}
\subfigure[Variance]
{ \label{fig:subfig:Var_T_limit}
\includegraphics[width=2.4in]{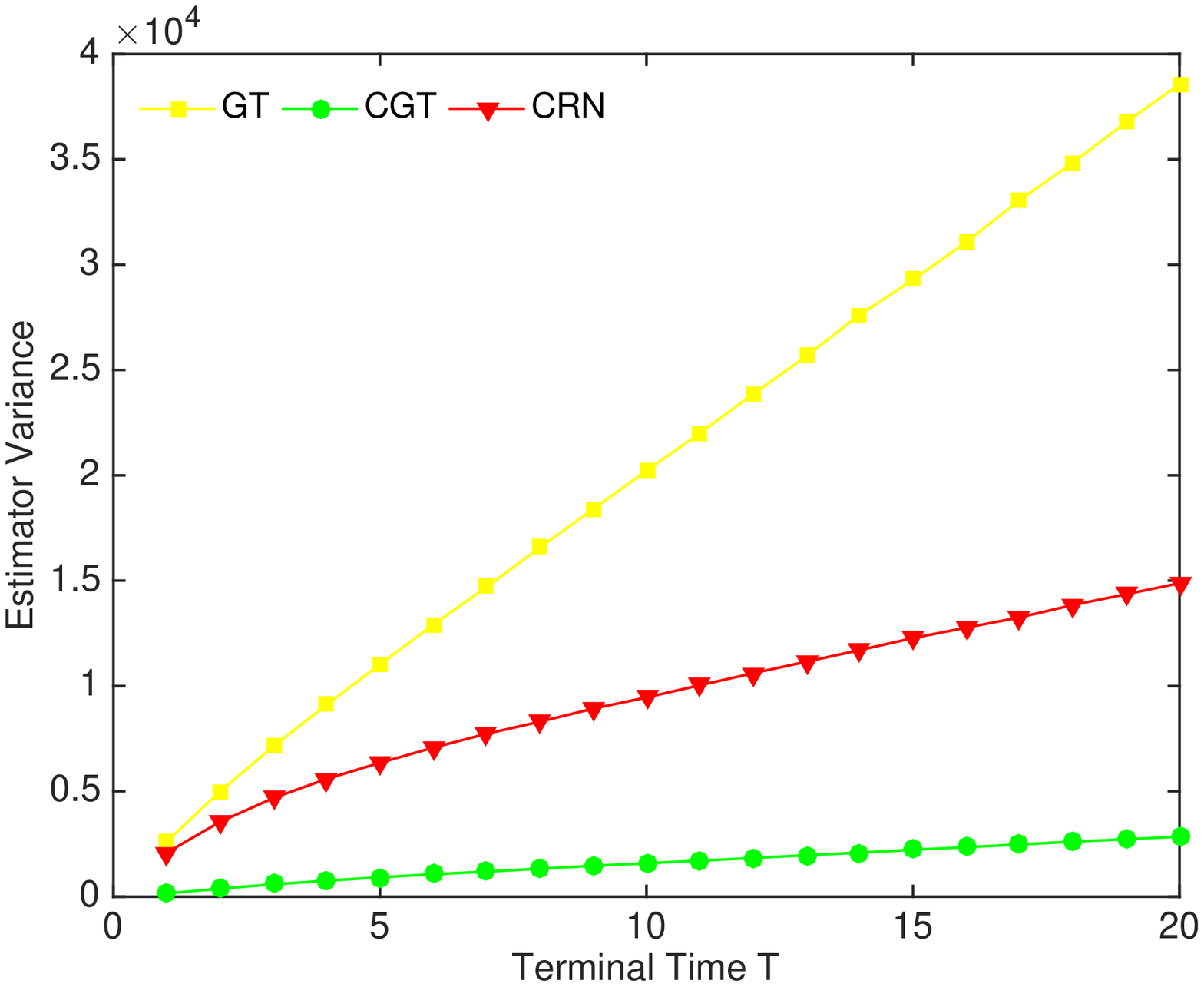}}
\caption{Estimated and true sensitivities (left) of $\mathbb{E}X_1(t)$ with respect to $c_1$ and the estimator variances (right) for the reversible isomerization model. The terminal time $T$ ($x$-axis) ranges from $1$ to $20$.}
\label{fig:T_limit}
\end{figure}

%
\section{Discussion and concluding remarks}\label{sec-discuss}
Our primary goal in this paper was to provide an analytical explanation of the phenomenon of
larger estimator variance of the GT method compared to the FD (as well as RPD
in the context of chemical kinetics) methods reported
frequently in the literature \cite{Glynn-book, Gir, CRP, RPD}. This was 
accomplished by our analysis in terms of system size $N$. 
The system size $N$ was taken to be proportional to system volume in the context
of stochastic chemical kinetics. 
Our analysis showed that the relative standard deviation (RSD) (see
\eqref{eq-RSD-def} for definition) of the GT, CGT and FD sensitivity
estimators are 
$\mathcal{O}(N^{1/2})$, $\mathcal{O}(1)$ and $\mathcal{O}(N^{-1/2})$,
respectively, as $N \to \infty$.  
The numerical examples provided also illustrate this point. Additionally, our
numerical examples also suggested that the RSD of the RPD method also scales
as $\mathcal{O}(N^{-1/2})$.  
We also showed that the relative bias (RB) (see \eqref{eq-RB-def} for definition) of any FD method 
was asymptotically $\mathcal{O}(1)$ as $N \to \infty$. 
We note that, in our analysis of the FD methods, we kept $h$ fixed and considered $N \to
\infty$ limit. Now we discuss, at least in theory, how $h$ may be chosen in terms of system size
$N$ to obtain the best performance for the FD methods.

{\bf Number of simulations required to achieve a given relative error (RE):}
Since the FD methods are biased while the GT and CGT methods are not, we shall use 
the relative error (RE) to compare the efficiencies of the GT, CGT and FD
estimators. More precisely, we shall estimate the number of trajectory 
simulations $N_s$, required to achieve a given 
tolerance $\delta$ for the relative error (RE) in the mean square sense 
which includes RB and RSD (see \eqref{eq-RE} for the exact definition).  

Our analysis for the FD methods was carried out so that large $N$ 
behavior for fixed $h$ was obtained.  We may combine our large $N$ analysis 
with small $h$ behavior of the FD methods already studied in the literature
\cite{Glynn-book}. In general, the bias of the one-sided FD estimator 
is $\mathcal{O}(h)$ as $h \to 0$, so we may expect the relative bias of an FD 
estimator to be given by $\text{RB} \approx C_2 h$ for small $h$ and large $N$,
where $C_2$ does not depend on $N$ or $h$. If higher order FD is used, then one 
expects $\text{RB} \approx C_2 h^{\gamma_1}$, where $\gamma_1 \geq 1$ in general. 
For instance, for the two-sided FD estimator we have that $\gamma_1=2$. 

Moreover, when using the independent random
number (IRN) FD method, the variance is $\mathcal{O}(1/h^2)$ as $h \to 0$,
which is similar to behavior of the upper bound used in our proof of Theorem 
\ref{thm:rate-FD}. However, when using common random number (CRN) FD methods,
one may typically expect $\mathcal{O}(1/h)$ dependence \cite{Glynn-book,CRP}.  
This is because, $\text{Cov}(f(X(t,c+h)),f(X(t,c)))$ is typically  
$\mathcal{O}(h)$ as $h \to 0$. 
Hence we may write $\text{RSD}^2 \approx C_1/(N h^{\gamma_2})$ for small $h$ 
and large $N$, where typically $\gamma_2 = 1$ or $2$
depending on whether CRN or IRN is used, and $C_1$ is independent of $h$ and
$N$. Combining the bias and the variance, and using \eqref{eq-RE-RSD-RB}, we
expect that, for an FD method
\begin{equation}
\text{RE}^2 = \frac{\text{RSD}^2}{N_s} + \text{RB}^2 \approx \frac{C_1}{N_s N
  h^{\gamma_2}} + C_2^2 h^{2 \gamma_1}.
\end{equation}
At this point, we must remark that in order for the above approximation to hold rigorously, one must establish the joint limit as $(N, h)\to (\infty, 0)$.
We believe that this could be done under additional regularity assumptions, but we shall not pursue this in this paper.

Extending the idea in \cite{Glynn-book} to include system size $N$, we look 
for the optimal choice of $h$ (the one that minimizes RE), for a given system size $N$ and number of
simulations $N_s$.  With some effort, one can see that the optimal $h$ is
given by $$h \propto N^{\frac{-1}{2 \gamma_1 + \gamma_2}} N_s^{\frac{-1}{2
    \gamma_1 + \gamma_2}},$$
and hence the minimal square RE for an FD method has the proportionality
\begin{equation}
\text{RE}^2 \propto N^{\frac{-2 \gamma_1}{2 \gamma_1 + \gamma_2}}
N_s^{\frac{-2 \gamma_1}{2 \gamma_1 + \gamma_2}}.
\end{equation}

On the other hand, for the CGT method, $\text{RE}^2 = \text{RSD}^2/N_s =
C_3/N_s$ for large $N$ where $C_3$ is independent of $N$ and $N_s$. 
Likewise, for the GT method, $\text{RE}^2 = \text{RSD}^2/N_s =
C_4 N/N_s$, for large $N$, where $C_4$ is independent of $N$ and $N_s$. 
Hence, for a specified value of $\delta$ for RE and a given system size $N$, the 
number of simulations required for the different methods are given by  
\begin{equation}
N_s^{\text{FD}} \propto \delta^{-2 - \frac{\gamma_2}{\gamma_1}} N^{-1},\;\;\;\;
N_s^{\text{CGT}}\propto \delta^{-2}, \;\;\;\;
N_s^{\text{GT}} \propto N \delta^{-2}. 
\end{equation}
We note that, as observed in \cite{Glynn-book}, the optimal dependence of
$N_s$ on $\delta$, 
is $\delta^{-2}$, which is achieved for an unbiased method. The biased FD
methods have suboptimal dependence on $\delta$, unless $\gamma_2=0$, which
is typically not the case in the context of discrete state systems, as $\gamma_2=0$ 
implies the validity of the (unregularized) pathwise derivative method \cite{Glynn-book}. However, when $N$ is much larger than $\delta^{-\gamma_2/\gamma_1}$, we expect the FD method to be more efficient than the CGT or GT. For instance, for $\delta
= 0.01$, if $N \gg 10$, say $N=50$ for instance,
we may expect the two-sided CRN FD ($\gamma_1=2,\gamma_2=1$) to be more efficient
than CGT which will be more efficient than GT. If one-sided CRN FD is used
($\gamma_1=\gamma_2=1$) or two-sided IRN FD is used ($\gamma_1=\gamma_2=2$), we expect FD to be more efficient when $N \gg 100$, 
say $N = 500$. If one-sided IRN FD is used ($\gamma_1=1,\gamma_2=2$) we expect
FD to be more efficient than CGT only for $N \gg 10^4$.   

Since the constants of proportionality that appear in the above discussion 
are not known in practice and typically harder to estimate than the
sensitivity itself, one may not expect to choose $h$ in a straightforward
manner based on the above discussion. Nevertheless, the above discussion
provides some idea of the optimal efficiency that could be expected.  

We also note that the comparison of an unbiased estimator with a biased one 
is more nuanced and qualitative. This is because, while one can estimate the variance of an estimator 
from the simulation, its bias cannot be estimated reliably unless one knows the exact 
quantity to be estimated! As a consequence, an unbiased
estimator is preferable to a biased one, unless the unbiased estimator has
exceedingly larger variance compared to the biased one. 
In this context, we 
like to mention that Multilevel Monte Carlo approaches (see
\cite{anderson2014complexity} for instance) may be used to combine a
biased low variance estimator with an unbiased high variance estimator to
obtain an efficient and unbiased estimator.

{\bf Factors other than system size that affect the RSD:} We note that factors other than system size also affect the RSD of an
estimator. One factor to study will be the dependence on $t$ as $t \to
\infty$. Our numerical simulations showed linear growth in $t$ behavior for GT,
CGT and even for FD methods for a practical range of $t$ values (up to a few
multiples of the time to stationarity). However, from a simple upper bound 
for the variance of the FD methods, we expect this growth to reach a finite maximum, 
for systems that are ergodic. The $\mathcal{O}(t)$ behavior (as $t
\to \infty$) for the variance of the GT and CGT methods 
can be justified theoretically, as explained in Section 5.3. 
Thus, dependence on time does not explain the greater variance of GT 
compared to CGT.

{\bf Extension of the variance analysis:}  
Our analysis made special use of the deterministic limit in the large system
size under what is known as the classical scaling which was used by 
Kurtz \cite{Ethier-Kurtz}. In other words,
 after suitable scaling, $f^N(X^N(t))$ converges to the deterministic limit
$f(X(t))$ almost surely. However, the scaled weight processes
$Z^N(t)/\sqrt{N}$ converge weakly to a Gaussian process $U(t)$. Our analysis
combined the two limits to obtain the desired results. 
Our results were proven under Assumptions 1-5 stated in Section 2. 
The first assumption assumes that the parameters enter multiplicatively 
: $a_j(x,c) = c_j b_j(x)$. This is satisfied by the stochastic mass action
form of intensities. In some literature on chemical kinetics, there
are some other forms of intensity functions that are used. 
Relaxing Assumption 1 to a general form will make the weight process $Z^N$ 
more complicated, and it will be given by a stochastic integral where both 
the integrand and the integrator are stochastic processes indexed by $N$. 
To obtain convergence of $N^{-1-2\alpha}
\bfE[(f^N(X^N(t)))^2 (Z^N(t))^2]$ one may need the result from
\cite{kurtz1991weak} which analyzes the limit of a sequence of stochastic integrals. We speculate that Assumption 4 may be relaxed using
stopping time arguments and sufficient integrability assumptions on the
process.   

In many practical systems some species are present in small numbers while
others are present in large numbers, and some reaction parameters are much 
larger than the others making the system ``stiff''. The classical scaling
studied here does not capture this. The more general scaling proposed in
\cite{Multiscale, sep-timescale} (again by Kurtz and collaborators) 
involve introducing a parameter $N$ which appears with different exponents 
both in the stochastic parameters $c'_j$ as well as the scaling of species 
and time itself. These analyses often provide stochastic limits to the 
scaled processes $X^N$. One could extend our current analysis along these lines 
to explore more subtle dependencies of the estimator variances. 
A related earlier work which scales all ``species'' by
the same factor $\epsilon$, and scales time differently $\epsilon^{-\alpha}$, 
in the context of processes driven by Levy measure can be found in
\cite{tomisaki1992homogenization}.

\section*{Acknowledgement}
We would like to thank the anonymous referees for the comments that helped improve the manuscript.


\bibliographystyle{siam}
\bibliography{mybib}
\end{document}